\documentclass[12pt]{amsart}
\usepackage{amssymb,amsmath,amsfonts,latexsym, amscd,amsthm}
\usepackage{verbatim}
\usepackage[colorlinks=true, linkcolor=blue, citecolor=blue, urlcolor=blue]{hyperref}

\newtheorem{theorem}{Theorem}[section]
\newtheorem{lemma}[theorem]{Lemma}
\newtheorem{corollary}[theorem]{Corollary}
\newtheorem{proposition}[theorem]{Proposition}
\theoremstyle{definition}

\newtheorem{example}[theorem]{Example}

\theoremstyle{remark}
\newtheorem{remark}[theorem]{Remark}

\def \D{\mathbb{D}}
\def \R{\mathbb{R}}
\def \Z{\mathbb{Z}}

\def \T{\mathbb{T}}

\newcommand{\clb}{\mathcal{B}}

\newcommand{\clr}{\mathcal{R}}

\newcommand{\cle}{\mathcal{E}}

\newcommand{\cln}{\mathcal{N}}

\newcommand{\clk}{\mathcal{K}}
\newcommand{\cls}{\mathcal{S}}

\newcommand{\cld}{\mathcal{D}}

\newcommand{\na}{\mathcal{NA}}
\newcommand{\vp}{\varphi}

\newcommand\restr[2]{\ensuremath{\left.#1\right|_{#2}}}

\numberwithin{equation}{section}

\subjclass[2020]{47B35, 47B20, 47A30, 30H10.}
\keywords{Dual truncated Toeplitz operator; norm attaining operator; Hardy space; Toeplitz operator; Hankel operator; extremal vector.}

\begin{document}
\title[Norm attaining dual truncated Toeplitz operators]{Norm attaining dual truncated Toeplitz operators}
\author{Sudip Ranjan Bhuia}
\address{Department of Mathematics, Shiv Nadar University, NH91, Tehsil Dadri, Greater Noida, Gautam Buddha Nagar, Uttar Pradesh-201314, India}
\email{sudipranjanb@gmail.com; sudip.bhuia@snu.edu.in}

\author{Puspendu Nag}
\address{Department of Mathematics, IIT Hyderabad, 
Kandi, Sangareddy, Telangana, India, 502\,284}
\email{ma23resch11003@iith.ac.in; puspomath@gmail.com}


\begin{abstract}
In this paper, we investigate norm attainment for dual truncated Toeplitz operators $D_\vp$ acting on $\clk_u^\perp=uH^2\oplus H^2_{-}$, where $u$ is a nonconstant inner function and $\vp\in L^\infty(\T)$. Our main focus is the structure of extremal vectors and the distinction between global and componentwise norm attainment.
	
For arbitrary $\vp\in L^\infty(\T)$, we establish an exact norm-defect identity and characterize the extremal space in terms of the essential maximum set $E_\vp=\{\zeta\in\T:|\vp(\zeta)|=\|\vp\|_\infty\}$. As a consequence, when $u$ is a finite Blaschke product,
\[
D_\vp\in\mathcal{NA}\quad\Longleftrightarrow\quad
m(E_\vp)>0.
\]
In this case, whenever $D_\vp$ is norm attaining, its extremal space is infinite-dimensional. We further show that the sets of symbols generating norm attaining and non-norm attaining DTTOs are both norm dense in $L^\infty(\T)$. Consequently, both the norm attaining and the non-norm attaining DTTOs are operator-norm dense in the class of all DTTOs associated with $u$.
	
For unimodular symbols, we characterize extremality by the condition
$M_\vp f\in\clk_u^\perp$ and equivalently by a truncated Hankel kernel condition. For mixed extremal vectors $f=x\oplus y$, we derive the identity
\[
\|D_\vp x\|^2-\|x\|^2=\|D_\vp y\|^2-\|y\|^2=-\langle D_\vp x,D_\vp y\rangle,
\]
which yields a phase-rotation criterion and coupled
Toeplitz--Hankel relations. We also show that global norm attainment
may occur even when neither the analytic nor the coanalytic component
contains a nonzero extremal vector. Under additional Hardy-space
hypotheses, we obtain factorization criteria for componentwise extremals, construct explicit extremal families for quotient-inner
symbols, and relate norm attainment of Toeplitz operators to that of
dual truncated Toeplitz operators.
\end{abstract}

\maketitle

\section{Introduction}
Toeplitz and Hankel operators on the Hardy space $H^2$ are among the
basic operator-theoretic objects associated with multiplication by
functions on the unit circle; see, for example,
\cite{Bercovivi,Duren,Garnett,Peller:Book}. Let $u$ be a nonconstant
inner function and let $\clk_u=H^2\ominus uH^2$ be the corresponding model space. The truncated Toeplitz operator with symbol $\vp$ is the compression of the multiplication operator $M_\vp$ to $\clk_u$. Truncated Toeplitz operators were introduced by Sarason \cite{Sarason2007} and have since been studied extensively; see, for example, \cite{Camara2020,Garcia2016} and the references therein.

For a non constant inner function $u$, consider the orthogonal complement $\clk_u^\perp=uH^2\oplus H^2_{-}$ of the model space $\clk_u$ in $L^2(\T)$, the \emph{dual truncated Toeplitz operator} with symbol $\vp\in L^\infty(\T)$ is defined by
\[
D_\vp=\restr{(I-P_{\clk_u})M_\vp}{\clk_u^\perp}.
\]
Dual truncated Toeplitz operators were introduced and studied by
Ding and Sang \cite{Sang:Ding}. Their Toeplitz--Hankel block
representation and Brown--Halmos type properties were developed
further in \cite{Sang:Qin:Ding}; see also
\cite{BhuiaRameshNag,Gu,Li:Sang:Ding,Wang:Zhao:Zheng}. The special
case $\vp(z)=z$, called the dual compressed shift, was studied by
Camara and Ross \cite{Camara2021}. In particular, when $u(0)=0$,
the dual compressed shift is unitarily equivalent to the direct sum
of the unilateral shift and its adjoint.

The purpose of this paper is to study norm attainment for dual
truncated Toeplitz operators. Recall that a bounded operator $T$ on a
Hilbert space $H$ is norm attaining if there exists a nonzero vector $x\in H$ such that $\|Tx\|=\|T\|\|x\|$. The study of norm attaining operators has a long history, beginning with the Bishop--Phelps theorem and its subsequent refinements; see, for example,
\cite{Lindenstrauss1963,ACOSTA,AL,CASCALES}. Questions concerning approximation and perturbation by norm attaining
operators have been studied in \cite{Enflo,KOV1,KOV2,LEE}, while
Bishop--Phelps--Bollob\'as type results for operators were developed
further in \cite{CHOI}. In the setting of
Toeplitz operators, norm attainment is closely connected with the
structure of the symbol. Yoshino \cite{Yoshino2002} proved that if
$T_\vp$ is norm attaining, then, after normalization, its symbol is of the form $\Theta_1\overline{\Theta_2}$ for inner functions
$\Theta_1$ and $\Theta_2$. Related rigidity phenomena for Toeplitz operators appear in the characterization of partially isometric Toeplitz operators due to Brown and Douglas \cite{Brown:Douglas}. A key ingredient in their analysis is Beurling's theorem \cite{Beurling1948}, which characterizes the nontrivial closed shift-invariant subspaces of $H^2$ as the spaces $uH^2$, where $u$ is inner. This invariant-subspace structure is also
the natural source of the model spaces $\clk_u=H^2\ominus uH^2$ that underlie truncated and dual truncated Toeplitz operators.
These classical rigidity results naturally motivate the study of norm attainment for dual truncated Toeplitz operators $D_\vp$ defined on $\clk_u^\perp$.

A first difference appears already at the level of the decomposition
$\clk_u^\perp=uH^2\oplus H^2_{-}$. Suppose that
$\|\vp\|_\infty=1$. We show that if $D_\vp$ has a nonzero extremal
vector lying entirely in either $uH^2$ or $H^2_{-}$, then
$|\vp|=1$ almost everywhere on $\T$. On the other hand, global norm
attainment of $D_\vp$ does not imply that $\vp$ is unimodular. We
give examples in which an extremal vector has both analytic and
coanalytic components although neither component separately contains
a nonzero extremal vector. Thus norm attainment on
$\clk_u^\perp$ is genuinely different from componentwise norm
attainment.

For unimodular symbols, the extremal vectors admit a simple
description. We prove that
\[
f\in\clk_u^\perp \text{ is extremal for }D_\vp
\quad\Longleftrightarrow\quad
M_\vp f\in\clk_u^\perp,
\]
or, equivalently, that the extremal space is the kernel of a
corresponding truncated Hankel operator. This description also
clarifies the interaction between the two components of an extremal
vector. If $f=x\oplus y$ is extremal, where
$x\in uH^2$ and $y\in H^2_{-}$, then
\[
\|D_\vp x\|^2-\|x\|^2
= \|D_\vp y\|^2-\|y\|^2 =-\langle D_\vp x,D_\vp y\rangle.
\]
Thus the losses of norm of the two individual components are exactly
balanced by their mixed term. From this identity we obtain a
phase-rotation criterion for mixed extremal vectors. We also express
extremality through the maps
$T_{+}h=P_{\clk_u}(\vp uh),\, h\in H^2$ and
$T_{-}g=P_{\clk_u}(\vp g),\,g\in H^2_{-}$: the vector $uh\oplus g$ is extremal
precisely when $T_{+}h+T_{-}g=0$. In addition, the block representation
of $D_\vp$ yields a coupled Toeplitz--Hankel system for vectors
satisfying $D_\vp D_\vp^*f=f$.

We then consider extremal vectors supported on one component.
These results require an additional Hardy-space condition. In the
analytic case, for $\|\vp\|_\infty=1$, we show that there exist inner
functions $\psi_{+}$ and $\chi_+$ such that
$\vp=\bar{u}\,\overline{\psi_{+}}\chi_+$ if and only if there exists a
nonzero analytic extremal $uh_0$ satisfying
$\vp uh_0\in H^2$. The corresponding coanalytic result is obtained
using the canonical conjugation $C_u$. We also describe the
componentwise extremal spaces under these factorizations and obtain
explicit infinite-dimensional families of extremal vectors for
quotient-inner symbols.

The preceding results concern primarily symbols of constant modulus.
For a general symbol $\vp\in L^\infty(\T)$, we obtain an exact
norm-defect identity which leads to a description of the entire
extremal space. If
\[
E_\vp=\{\zeta\in\T:|\vp(\zeta)|=\|\vp\|_\infty\},
\]
then every extremal vector must be supported on $E_\vp$, together
with an additional projection condition determined by the model
space. This gives a complete criterion when $u$ is a finite Blaschke
product:
\[
D_\vp\in\mathcal{NA}
\quad\Longleftrightarrow\quad
m(E_\vp)>0.
\]
Moreover, whenever norm attainment occurs, the extremal space is
infinite-dimensional. More precisely, if $u$ has degree $n$, the
extremal space has finite codimension, bounded by $2n$, inside the
corresponding $L^2(E_\vp)$ space. We also obtain the componentwise
analogue and show that global norm attainment may occur even when
both componentwise extremal spaces are trivial.

The finite Blaschke case also exhibits an interesting approximation
phenomenon. We prove that both
\[
\{\vp\in L^\infty:D_\vp\in\mathcal{NA}\}
\quad\text{and}\quad
\{\vp\in L^\infty:D_\vp\notin\mathcal{NA}\}
\]
are norm dense in $L^\infty$. Since
$\|D_\vp-D_\psi\|=\|\vp-\psi\|_\infty$, it follows that both the
norm attaining and the non-norm attaining DTTOs are operator-norm
dense in the class $\{D_\vp:\vp\in L^\infty\}$.

Finally, we relate norm attainment for DTTOs to the classical
Toeplitz problem. Yoshino's factorization theorem yields
componentwise extremal vectors for $D_\vp$ whenever $T_\vp$ is norm
attaining. Conversely, for analytic symbols, norm attainment of
$D_\vp$ on $uH^2$ implies norm attainment of $T_\vp$. The
corresponding coanalytic statement follows by conjugation.

The paper is organized as follows. Section~2 collects the necessary
background on Hardy and model spaces, Toeplitz and Hankel operators,
dual truncated Toeplitz operators, and their block representations.
Section~3 studies the constant-modulus case. We characterize
extremal vectors for unimodular symbols, analyze the interaction
between their analytic and coanalytic components, derive the coupled
Toeplitz--Hankel relations, and prove the componentwise factorization
results. Section~4 treats arbitrary $L^\infty$ symbols. We establish
the norm-defect identity, describe the extremal space through the
essential maximum set, obtain the complete finite-Blaschke criterion,
and prove the density results. Section~5 compares norm attainment for
Toeplitz operators and dual truncated Toeplitz operators.

\section{Preliminaries} 

We first fix some standard notation and recall basic facts about Hardy spaces, model spaces, and (truncated) Toeplitz-type operators; see \cite{Garcia2016} for a general reference. Let $\D$ be the open unit disk and $\T$ be the unit circle in the complex plane. Let $m$ denote the normalized Lebesgue measure on $\T$, given by
\[
dm(e^{i\theta})=\frac{d\theta}{2\pi},
\qquad 0\leq\theta<2\pi.
\]
We write $L^2(\T)$ for the Hilbert space of square-integrable complex-valued functions on $\T$ with respect to $m$, and abbreviate
\[
L^2:=L^2(\T).
\]
For a measurable set $E\subseteq\T$, $m(E)$ denotes its normalized Lebesgue measure.
\subsection{Hardy space}
Let $H^2:=H^2(\T)$ denote the Hardy space, identified with the closed
subspace of $L^2$ consisting of functions whose negative Fourier coefficients vanish. Thus, if
\[
f(e^{it})=\sum_{n\in\Z}c_ne^{int},
\]
then $f\in H^2$ if and only if $c_n=0$ for all $n<0$. We have the
orthogonal decomposition
\[
L^2=H^2\oplus H^2_{-},
\qquad
H^2_{-}:=L^2\ominus H^2
=\overline{\operatorname{span}}\{e^{-int}:n\ge1\}.
\]
We write $P_{+}$ for the orthogonal projection of $L^2$ onto $H^2$,
let $L^\infty:=L^\infty(\T)$ denote the space of essentially bounded
measurable functions on $\T$, and set $H^\infty:=H^2\cap L^\infty$.

For $\vp\in L^\infty$, let $M_\vp$ denote multiplication by $\vp$ on
$L^2$. The Toeplitz, Hankel, and dual Toeplitz operators with symbol
$\vp$ are given by
\[
T_\vp=\restr{P_{+}M_\vp}{H^2},
\quad
H_\vp=\restr{(I-P_{+})M_\vp}{H^2},
\quad\text{and}\quad
S_\vp=\restr{(I-P_{+})M_\vp}{H^2_{-}},
\]
respectively.

\subsection{Model spaces and truncated Toeplitz operators}

Let $u$ be a fixed nonconstant inner function, that is, $u\in H^\infty$ with $|u(e^{it})|=1$ a.e.\ on $\T$. The associated \emph{model space} is
\[
\clk_u := H^2 \ominus uH^2 = H^2\cap(uH^2)^\perp.
\]
This is a closed subspace of $H^2$, and it is finite dimensional when $u$ is a finite Blaschke product, but infinite dimensional in general (for example, when $u$ is an infinite Blaschke product or a singular inner function).

Let $P_{\clk_u}$ denote the orthogonal projection of $L^2$ onto $\clk_u$. For $\vp\in L^\infty$, the \emph{truncated Toeplitz operator} (TTO) on $\clk_u$ is defined by
$$
A_\vp := P_{\clk_u} \restr{M_\vp}{\clk_u}.
$$
When $\vp(z)=z$, the corresponding operator is called \emph{compressed shift} and is denoted by $S_u$.

For a nonconstant inner function $u$, the orthogonal complement of $\clk_u$ in $L^2$ is given by
$$
\clk_u^\perp = uH^2 \oplus H^2_{-}.
$$
Thus any $h\in \clk_u^\perp$ can be written as $h = uf + \overline{zg}$ for some $f,g\in H^2$. In particular, we have
$$
\clk_u\oplus \clk_u^\perp = L^2
= H^2\oplus H^2_{-}
= \clk_u\oplus uH^2\oplus H^2_{-}.
$$
The orthogonal projection of $L^2$ onto $\clk_u^\perp$ can be written as
\begin{equation}\label{eq:ortho-model-projection}
P_{\clk_u^\perp} := I - P_{\clk_u} = I - P_{+} + M_u P_{+} M_{\overline{u}}.
\end{equation}

\subsection{Dual truncated Toeplitz and truncated Hankel operators}

The \emph{dual truncated Toeplitz operator} (DTTO) on $\clk_u^\perp$ with symbol $\vp\in L^\infty$ is defined by
\[D_\vp := (I-P_{\clk_u})\restr{M_\vp}{\clk_u^\perp}.\]
Using \eqref{eq:ortho-model-projection}, we obtain the concrete formula
\begin{equation}\label{eq:Dtphi-explicit}
D_\vp(h) = (I-P_{+})(\vp h) + u\,P_{+}(\bar{u}\vp h),\qquad h\in \clk_u^\perp.
\end{equation}
In particular, $(I-P_{+})(\vp h)\in H^2_{-}$ and $M_u P_{+}M_{\bar{u}}\vp h\in uH^2\subset H^2$, so these two terms are mutually orthogonal.

In the special case $\vp(z)=z$, the operator $D_\vp$ is referred to as the \emph{dual compressed shift}.

The following elementary observations will be useful:

\begin{enumerate}
\item For every $g\in H^2$, we have $ug\in uH^2\subset\clk_u^\perp$.
Hence, by \eqref{eq:Dtphi-explicit},
\[
\|D_\vp(ug)\|^2 =\|(I-P_{+})(\vp ug)\|^2+\|M_uP_{+}M_{\bar{u}}(\vp ug)\|^2.
\]
	
\item For every $g\in H^2$,
\[
\|T_\vp g\|^2=\|M_uT_\vp M_{\bar{u}}(ug)\|^2
=\|M_uP_{+}M_{\bar{u}}(\vp ug)\|^2
\leq\|D_\vp(ug)\|^2,
\]
where the last inequality follows from \textup{(i)}.
\end{enumerate}

For $\vp\in L^\infty$, the big \emph{truncated Hankel} operator $B_\vp:\clk_u\to\clk_u^\perp$ is defined by
\begin{equation}\label{THO}
B_\vp(f) = (I-P_{\clk_u})\vp f, \qquad f\in \clk_u.
\end{equation}
Its adjoint is given by
\begin{equation}
B^*_\vp(f) = P_{\clk_u}(\bar{\vp} f),\qquad f\in \clk_u^\perp.
\end{equation}

With respect to the decomposition $L^2=\clk_u\oplus \clk_u^\perp$, the multiplication operator $M_\vp$ admits the block matrix representation
$$
M_\vp =
\begin{bmatrix}
	A_\vp & B^*_{\bar{\vp}} \\
	B_{\vp} & D_\vp
\end{bmatrix}
\quad\text{on } \clk_u \oplus \clk^\perp_u.
$$
From the relation $M_\vp M_\psi = M_{\vp\psi}$, we obtain
\begin{align}
B^*_{\bar{\vp}} B_\psi &= A_{\vp \psi} - A_\vp A_\psi, \\
B_\vp B^*_{\bar{\psi}} &= D_{\vp \psi} - D_\vp D_\psi, \label{DTTO-THO-relation}\\
B_\vp A_\psi &= B_{\vp \psi} - D_\vp B_\psi.
\end{align}

\subsection{Conjugation and basic properties of $D_\vp$}

For a nonconstant inner function $u$, the canonical conjugation $C_u : L^2 \longrightarrow L^2$ is given by
\[C_u f(e^{it}) := u(e^{it})\,e^{-it}\,\overline{f(e^{it})},\quad f\in L^2.\]
It satisfies $C_u^2 = I$, $\langle C_u f, C_u g \rangle = \langle g, f \rangle$, and $\|C_u f\| = \|f\|$ for all $f,g\in L^2$. A key feature is that $C_u$ exchanges the analytic and coanalytic components of the decomposition
\[\clk_u^\perp = uH^2 \oplus H^2_{-}.\]
More precisely,
\begin{equation}\label{eq:Cu-swap-cor}
C_u(uH^2) = H^2_{-}, \qquad C_u(H^2_{-}) = uH^2.
\end{equation}
Next, we recall some basic properties of DTTOs.

\begin{proposition}\cite{Sang:Ding}
The following are true:
\begin{enumerate}
\item $D_\vp$ is bounded if and only if $\vp\in L^\infty$. Moreover, $\|D_\vp\|=\|\vp\|_\infty$.
\item $D_\vp$ is compact if and only if $\vp=0$ a.e.\ on $\T$.
\item $D^*_\vp = D_{\bar{\vp}}$.
\end{enumerate}
\end{proposition}
It is well known that $D_\vp$ is $C_u$--symmetric \cite[Theorem 2.7]{Gu}, that is,
\begin{equation}\label{eq:Cu-symm-cor}
D_\vp^* = C_u D_\vp C_u.
\end{equation}
This identity will later allow analytic norm attainment results for $D_{\vp}$ to be transferred directly to coanalytic norm attainment results for $D_{\bar{\vp}}$.
\subsection{Block representation of $D_\vp$}

Define a unitary operator $U:L^2 = H^2\oplus H^2_{-}\to \clk_u^\perp = uH^2\oplus H^2_{-}$ by
\begin{equation}\label{unitary:equiv}
U =
\begin{bmatrix}
M_u & 0\\
0   & I
\end{bmatrix}.
\end{equation}
Then we have the following block representation:

\begin{lemma}\cite{Sang:Qin:Ding}\label{block-rptn}
Let $\vp\in L^\infty$. Then
\[U^*D_\vp U= \begin{bmatrix}
	T_\vp & H^*_{u\bar{\vp}}\\
	H_{u\vp} & S_\vp
\end{bmatrix}\]
on the space $L^2 = H^2\oplus H^2_{-}$, where the unitary operator $U$ is given by \eqref{unitary:equiv}.
\end{lemma}

The anti-unitary operator $V$ on $L^2$ is given by
\begin{equation}
Vf(z)=\bar{z}\overline{f(z)}, \quad f\in L^2, z\in \T.
\end{equation}
Note that $V=V^{-1}$,\quad $VT_\vp=S_{\bar{\vp}}V$. In addition, $V(H^2)=H^2_{-}$ and $V(H^2_{-})=H^2$, see \cite{Guediri} for details.

\subsection{Norm attaining operators on Hilbert spaces}

We now recall the basic framework of norm attaining operators on a Hilbert space. Let $(H,\langle\cdot,\cdot\rangle)$ be a complex Hilbert space, and let $\clb(H)$ denote the algebra of all bounded linear operators on $H$. For $T\in\clb(H)$, we write
\[
\cln(T) := \{x\in H : Tx=0\},\qquad \clr(T) := \{Tx:x\in H\}
\]
for null space (kernel) and range spaces of $T$, respectively.
An operator $T\in\clb(H)$ is said to be \emph{norm attaining}, or $T\in\mathcal{NA}$, if there exists a nonzero vector $x\in H$ such that
$$
\|Tx\|=\|T\|\,\|x\|.
$$
By homogeneity of norm, this is equivalent to the existence of a unit vector $x\in H$ with $\|Tx\|=\|T\|$. Any such vector is called an \emph{extremal vector} (or \emph{norm attaining vector}) for $T$. We denote the set of all extremal vectors by
$$
\cle_T := \{\,x\in H :  \|Tx\|=\|T\|\,\|x\|\,\},
$$
which is a closed subspace of $H$. In fact, by \cite[Lemma 3.1]{Ramesh}, we have $\cle_T =\cln(\|T\|^2I - T^*T)$. For a DTTO, we abbreviate
\[
\cle_\vp:=\cle_{D_\vp}
=\left\{f\in\clk_u^\perp:\|D_\vp f\|=\|\vp\|_\infty\|f\|
\right\}.
\]
With respect to the orthogonal decomposition $\clk_u^\perp=uH^2\oplus H^2_{-}$, we shall also use the notation
\[
\cle_\vp^{(+)}
:=\cle_\vp\cap\left(uH^2\oplus\{0\}\right),
\qquad\cle_\vp^{(-)}:=
\cle_\vp\cap\left(\{0\}\oplus H^2_{-}\right).
\]
Thus $\cle_\vp^{(+)}$ and $\cle_\vp^{(-)}$ denote, respectively,
the analytic and coanalytic extremal subspaces of $D_\vp$. Throughout the paper, when we say that $D_\vp$ attains its norm
on the analytic component $uH^2$, we mean that
\[
\cle_\vp^{(+)}\neq\{0\},
\]
that is, there exists $0\neq f\in uH^2$ such that
\[
\|D_\vp f\|=\|D_\vp\|\,\|f\|
=\|\vp\|_\infty\|f\|.
\]
Similarly, $D_\vp$ attains its norm on the coanalytic component
$H^2_{-}$ means that $\cle_\vp^{(-)}\neq\{0\}$.
This terminology should not be confused with norm attainment of
the restrictions $\restr{D_\vp}{uH^2}$ or $\restr{D_\vp}{H^2_{-}}$ at their possibly smaller operator norms.

We summarize the following phenomenon, which will be used in the study of the norm attaining properties of DTTOs (see \cite[Corollary 2.4; Proposition 2.5]{Carvajal2012}):

\begin{theorem}\label{thm:NA_equiv}
For $T \in \clb(H)$, the following are equivalent:
\begin{enumerate}
\item $T \in \mathcal{NA}$;
\item $T^* \in \mathcal{NA}$;
\item $T^* T \in \mathcal{NA}$;
 \item $TT^* \in \mathcal{NA}$;
\item $\|T\|^2$ is a common eigenvalue of $T^* T$ and $TT^*$.
\end{enumerate}
\end{theorem}

\begin{lemma}
Let $A\in\clb(H_1)$ and $B\in\clb(H_2)$. Suppose that $A$ and $B$
are unitarily equivalent. Then $A$ is norm attaining if and only if $B$ is norm attaining.
\end{lemma}

\begin{proof}
Let $U:H_1\to H_2$ be unitary and suppose that $B=UAU^*$. If $0\neq x\in H_1$ is an extremal vector for $A$, then
\[
\|B(Ux)\|=\|UAx\|=\|Ax\|=\|A\|\|x\|=\|B\|\,\|Ux\|.
\]
Hence $Ux$ is an extremal vector for $B$, and therefore $B$ is norm attaining. The converse follows by applying the same argument to $U^*$.
\end{proof}

\begin{remark}
  For $\vp(z)=z$, by Lemma \ref{block-rptn}, we obtain
\[U^*D_z U =
\begin{bmatrix}
	T_z & H^*_{u\bar{z}}\\
	0   & S_z
\end{bmatrix}:= X \quad \text{on } H^2\oplus H^2_{-}.\]
We claim that the operator $D_z$ is norm attaining. Let $f\in H^2$ be a nonzero vector with $\|f\|=1$. Then
\[\left\|
X\begin{bmatrix}
	f\\[2pt] 0
\end{bmatrix}
\right\|
=
\left\|
\begin{bmatrix}
	T_z f\\[2pt] 0
\end{bmatrix}
\right\|
= \|f\| = 1 = \|X\|.\]
Thus $X\in \mathcal{NA}$, and by the lemma above, we conclude that $D_z$ is also norm attaining. 
\end{remark}

\subsection{The block operator viewpoint}

Via the unitary identification between $\clk_u^\perp$ and $ H^2 \oplus H^2_{-}$, Lemma~\ref{block-rptn} shows that $D_\vp$ is unitarily equivalent to the block operator matrix
\[
\cld_\vp
:=
\begin{bmatrix}
	T_\vp & H^*_{u\bar{\vp}} \\
	H_{u\vp} & S_\vp
\end{bmatrix}.
\]
Thus $D_\vp$ is norm attaining if and only if $\cld_\vp$ is norm attaining, and hence there exists a unit vector $f\oplus g\in H^2\oplus H^2_{-}$ such that
$$
\|\cld_\vp (f\oplus g)\| = \|\cld_\vp\|.
$$
Writing out the action,
\begin{equation}
\cld_\vp(f\oplus g)
=
\big(T_\vp f + H^*_{u\bar{\vp}}g\big)
\oplus
\big(H_{u\vp}f + S_\vp g\big),
\end{equation}
we obtain coupled extremal relation involving all four block entries of $\cld_\vp$. However, in general, norm attainment of the block operator does \emph{not} force any of the individual entries $T_\vp$, $H_{u\vp}$, $H^*_{u\bar{\vp}}$, $S_\vp$ to have trivial kernel or to satisfy any simple kernel condition.

\section{Norm attainment of $D_\vp$ on $\clk_u^\perp$}

A key link between dual truncated Toeplitz operators and block Toeplitz models arises when the inner function $u$ satisfies $u(0)=0$. By \cite[Corollary 4.3]{Camara2021}, we have the following unitary equivalence
$$
D_z \cong T_z \oplus T_z^*.
$$
This observation motivates us to first investigate the norm attaining property of block operators of the form $T_\vp \oplus T^*_\psi$.

Let $\vp,\psi\in L^\infty$ and consider
\[
T:=T_\vp\oplus T_\psi^*  :  H^2\oplus H^2 \longrightarrow H^2\oplus H^2.
\]
We have $\|T_\vp\|=\|\vp\|_\infty$, $\|T_\psi^*\|=\|T_\psi\|=\|\psi\|_\infty$, hence
\[
\|T\|=\max\{\|\vp\|_\infty,\|\psi\|_\infty\}.
\]

\begin{lemma}\label{lem:direct-sum-NA}
Let $A\in \clb(H_1)$ and $B\in\clb(H_2)$ with $\alpha:=\|A\|$, $\beta:=\|B\|$, and $\gamma:=\max\{\alpha,\beta\}=\|A\oplus B\|$. Then
\begin{enumerate}
\item\label{alpha>beta} If $\alpha>\beta$, then $A\oplus B$ is $\na$ if and only if $A$ is $\na$.
\item\label{beta>alpha} If $\beta>\alpha$, then $A\oplus B$ is $\na$ if and only if $B$ is $\na$.
\item\label{alpha=beta} If $\alpha=\beta$, then $A\oplus B$ is $\na$ if and only if at least one of $A,B$ is $\na$.
\end{enumerate}
Moreover, whenever $A\oplus B$ is $\na$, there is an extremal vector supported entirely in a maximal–norm summand.
\end{lemma}

\begin{proof}

\textbf{Proof of (\ref{alpha>beta}):} Suppose $\alpha>\beta$.  
If $A$ is $\na$ at a unit vector $x$, then $x\oplus 0$ is the extremal vector for $A\oplus B$. 

Conversely, if $A\oplus B$ is $\na$ at a unit vector $x\oplus y$, then
\begin{equation}\label{A-direct-B}
\|(A\oplus B)(x\oplus y)\|^2=\|A\oplus B\|^2=\alpha^2.
\end{equation}
If $\|y\|\neq 0$, we would get
$$
\|(A\oplus B)(x\oplus y)\|^2\le \alpha^2\|x\|^2+\beta^2\|y\|^2
< \alpha^2\big(\|x\|^2+\|y\|^2\big)=\alpha^2,
$$
which contradicts the above equality \eqref{A-direct-B}. Therefore, $\|y\|=0$, that forces $y=0$ and $\|x\|=1$. From the equation \eqref{A-direct-B}, it follows that $\|Ax\|=\alpha$, that is, $A$ is $\na$.  

\textbf{Proof of (\ref{beta>alpha}):} Proof is same as (\ref{alpha>beta}) by changing the role of $\alpha$ by $\beta$.

\textbf{Proof of (\ref{alpha=beta}):}
Assume that $\alpha=\beta$. If $A$ is $\na$ at some unit vector
$x_0\in H_1$, then $x_0\oplus0$ is an extremal vector for
$A\oplus B$. Similarly, if $B$ is $\na$ at some unit vector
$y_0\in H_2$, then $0\oplus y_0$ is an extremal vector for
$A\oplus B$.

Conversely, suppose that $A\oplus B$ is norm attaining at a unit
vector $x\oplus y\in H_1\oplus H_2$. Since
$\alpha=\beta=\|A\oplus B\|$, we have
\[
\alpha^2
=
\|Ax\|^2+\|By\|^2.
\]
On the other hand,
\[
\|Ax\|^2\leq\alpha^2\|x\|^2,
\qquad
\|By\|^2\leq\alpha^2\|y\|^2.
\]
Hence
\[
\begin{aligned}
0&=\alpha^2(\|x\|^2+\|y\|^2)-\left(\|Ax\|^2+\|By\|^2\right)\\
	&=\left(\alpha^2\|x\|^2-\|Ax\|^2\right)+\left(\alpha^2\|y\|^2-\|By\|^2\right).
\end{aligned}
\]
Both terms on the right-hand side are nonnegative, so each must be
zero. Therefore,
\[
\|Ax\|=\alpha\|x\|,
\qquad
\|By\|=\alpha\|y\|.
\]
Since $x\oplus y$ is nonzero, at least one of $x$ or $y$ is nonzero.
If $x\neq0$, then $x/\|x\|$ is an extremal vector for $A$; if
$y\neq0$, then $y/\|y\|$ is an extremal vector for $B$. Thus at
least one of $A$ and $B$ is norm attaining.

Hence, when $\alpha=\beta$, $A\oplus B$ is norm attaining if and
only if at least one of $A$ and $B$ is norm attaining.

\end{proof}

\begin{theorem}\label{thm:NA-direct-sum-Toeplitz}
Let $\vp,\psi\in L^\infty$. Then $T:=T_\vp\oplus T_\psi^*$ is $\na$ if and only if one of the following holds
\begin{enumerate}
\item[\textup{(a)}] $\|\vp\|_\infty>\|\psi\|_\infty$ and $T_\vp$ is $\na$ ;
\item[\textup{(b)}] $\|\psi\|_\infty>\|\vp\|_\infty$ and $T_\psi$ is $\na$ ;
\item[\textup{(c)}] $\|\vp\|_\infty=\|\psi\|_\infty$ and at least one of $T_\vp$, $T_\psi$ is $\na$.
\end{enumerate}
Moreover, an extremal vector for $T$ can be chosen as $f\oplus 0$ (case \textup{(a)}), $0\oplus g$ (case \textup{(b)}), or supported in a norm-attaining summand in case \textup{(c)}.
\end{theorem}

\begin{proof}
Apply Lemma~\ref{lem:direct-sum-NA} with $A=T_\vp$ and $B=T_\psi^*$, and use Theorem~\ref{thm:NA_equiv} to identify $\na$ of $T_\psi^*$ with $\na$ of $T_\psi$.
\end{proof}

Let $\vp\in L^\infty$ with $\|\vp\|_\infty=1$. Next, we show that norm attainment of $D_\vp$ on either analytic or coanalytic components forces unimodularity of the symbol.
\begin{proposition}\label{prop:componentwise-unimodular}
Let $\vp\in L^\infty$ with $\|\vp\|_{\infty}=1$. If $D_\vp$ is norm attaining on either $uH^2$ or $H^2_{-}$, then $|\vp|=1$ a.e. on $\T$.
\end{proposition}

\begin{proof}
Suppose first that $D_\vp$ is norm attaining on $uH^2$. Then there exists a nonzero $h\in H^2$ such that
\[
\|D_\vp(uh)\|=\|uh\|.
\]
Since $u$ is inner and $\|\vp\|_\infty=1$, we have
\[
\|h\|=\|uh\|=\|D_\vp(uh)\| =\left\|P_{\clk_u^\perp}(\vp uh)\right\|\leq \|\vp uh\|
=\|\vp h\|\leq \|\vp\|_\infty\|h\|=\|h\|.
\]
In particular,
\[
\|\vp h\|=\|h\|.
\]
Therefore,
\[
0=\|h\|^2-\|\vp h\|^2=\int_{\T}(1-|\vp|^2)|h|^2\,dm.
\]
Since $\|\vp\|_\infty=1$, we have $1-|\vp|^2\geq0$ a.e., and hence
\[
(1-|\vp|^2)|h|^2=0\qquad\text{a.e. on }\T.
\]
Since $h\in H^2$ is nonzero, its boundary function does not vanish on a set of positive measure (cf. \cite[Theorem 6.3]{Douglas1998}). Thus $|h|>0$ a.e. on $\T$, and consequently
\[
|\vp|=1\qquad\text{a.e. on }\T.
\]
	
Now suppose that $D_\vp$ is norm attaining on $H^2_{-}$. Then there exists a nonzero $f\in H^2_{-}$ such that $\|D_\vp f\|=\|f\|$. Since
\[
C_u(H^2_{-})=uH^2,
\quad\text{and}\quad
D_{\bar{\vp}}=D_\vp^*=C_uD_\vp C_u,
\]
we obtain
\[
\|D_{\bar{\vp}}(C_uf)\| =\|C_uD_\vp f\|=\|D_\vp f\|=\|f\|=\|C_uf\|.
\]
Thus $D_{\bar{\vp}}$ is norm attaining on $uH^2$. Applying the first part to $\bar{\vp}$ yields
\[
|\bar{\vp}|=1\qquad\text{a.e. on }\T,
\]
and therefore, $|\vp|=1$ a.e. on $\T$.
\end{proof}

For Toeplitz operators on $H^2$, norm attainment of $T_\vp$ with
$\|\vp\|_\infty=1$ implies that $|\vp|=1$ a.e. on $\T$.
For dual truncated Toeplitz operators, this need not be true.
Indeed, $D_\vp$ may attain its norm on $\clk_u^\perp
=uH^2\oplus H^2_{-}$ even when $\vp$ is not unimodular, as the
following example shows.

\begin{example}
Let $u(z)=z$. Then the corresponding model space is
\[
\clk_u = H^2 \ominus zH^2 = \operatorname{span}\{1\}.
\]

Consequently, the orthogonal complement is
\[
\clk_u^\perp = \{\,h\in L^2: \langle h,1\rangle=0\,\}
 = \left\{h\in L^2: \int_{\T}h\,dm=0\right\}.
\]

Now choose a measurable set $E\subset\T$ satisfying $0<m(E)<1$, and define
\[
\vp = \chi_E \in L^\infty(\T), \qquad \|\vp\|_\infty = 1.
\]
Since $\vp=0$ on $\T\setminus E$ and
$m(\T\setminus E)>0$, the equality
\[
|\vp|=1
\]
does not hold almost everywhere on $\T$. Hence $\vp$ is not
unimodular.

Next, choose disjoint measurable subsets $E_1, E_2 \subset E$ with $m(E_1),m(E_2)>0$,
and define
\[
h := \chi_{E_1} - \frac{m(E_1)}{m(E_2)}\,\chi_{E_2}.
\]
Then $h\neq 0$, because $E_1$ and $E_2$ have positive measure and are disjoint.
Now 
\[
\widehat{h}(0) = \int_{\T} h\,dm
= \int_{\T}\left(\chi_{E_1} - \frac{m(E_1)}{m(E_2)}\chi_{E_2}\right)dm
= m(E_1) - \frac{m(E_1)}{m(E_2)}m(E_2) = 0.
\]
From the definition of $\clk_u^\perp$, it follows that $h\in \clk_u^\perp$.

Since $\operatorname{supp} h\subseteq E$ and $\vp=\chi_E$, we have
$\vp h=h$. As $h\in\clk_u^\perp$, it follows that
\[
D_\vp h =P_{\clk_u^\perp}(\vp h)
=P_{\clk_u^\perp}h=h.
\]
Therefore,
\[
\|D_\vp h\|=\|h\|.
\]
Since $\|\vp\|_\infty=1$, we have $\|D_\vp\|=1$. Hence
\[
\|D_\vp h\|=\|D_\vp\|\,\|h\|,
\]
and thus $h$ is a nonzero extremal vector for $D_\vp$. In particular,
$D_\vp$ is norm attaining.

However, the symbol $\vp=\chi_E$ is not unimodular a.e. on $\T$. Therefore, unlike the classical norm-attainment phenomenon for
Toeplitz operators, norm attainment for a DTTO does not imply
global unimodularity of the symbol.

\end{example}

The following proposition characterizes the extremal vectors of
$D_\vp$ for a unimodular symbol in terms of multiplication by $\vp$.

\begin{proposition}\label{prop:extremal-set}
Let $\vp\in L^\infty$ with $|\vp|=1$ a.e.\ on $\T$. Define
\[
\cle_\vp=
\left\{h\in\clk_u^\perp:\|D_\vp h\|=\|h\|\right\}
\quad\text{and}\quad
C=\left\{h\in\clk_u^\perp:M_\vp h\in\clk_u^\perp\right\}.
\]
Then $C=\cle_\vp$. Consequently, $D_\vp$ is norm attaining if and only if $C\neq\{0\}$.
\end{proposition}

\begin{proof}
Since $|\vp|=1$ a.e.\ on $\T$, the multiplication operator
$M_\vp$ is unitary on $L^2$. Hence, for every $h\in\clk_u^\perp$,
\[
\|M_\vp h\|=\|h\|.
\]
Using the orthogonal decomposition $L^2=\clk_u\oplus\clk_u^\perp$
and the identity $D_\vp h=P_{\clk_u^\perp}M_\vp h$,
we obtain
\[
\|h\|^2=\|M_\vp h\|^2=\|P_{\clk_u}M_\vp h\|^2+\|P_{\clk_u^\perp}M_\vp h\|^2=\|P_{\clk_u}M_\vp h\|^2+\|D_\vp h\|^2.
\]
Therefore,
\[
\|D_\vp h\|=\|h\|
\]
if and only if
\[
P_{\clk_u}M_\vp h=0.
\]
Since $P_{\clk_u}M_\vp h=0$ if and only if $M_\vp h\in\clk_u^\perp$,
we conclude that
\[
\cle_\vp=C.
\]
	
Finally, since $|\vp|=1$ a.e.\ on $\T$,
\[
\|D_\vp\|=\|\vp\|_\infty=1.
\]
Hence $D_\vp$ is norm attaining if and only if $\cle_\vp$ contains a nonzero vector, equivalently, if and only if $C\neq\{0\}$.
\end{proof}

\begin{corollary}
Let $\vp\in L^\infty$ with $|\vp|=1$ a.e.\ on $\T$, and let
$\cle_\vp$ denote the extremal set of $D_\vp$. Then $\cle_\vp=\cln(B^*_{\bar{\vp}})$.
\end{corollary}

\begin{proof}
For $h\in\clk_u^\perp$, we have $B^*_{\bar{\vp}}h=P_{\clk_u}M_\vp h$.
By Proposition~\ref{prop:extremal-set},
\[
h\in\cle_\vp
\quad\Longleftrightarrow\quad
M_\vp h\in\clk_u^\perp.
\]
Since $M_\vp h\in\clk_u^\perp$ if and only if $P_{\clk_u}M_\vp h=0$,
we obtain $h\in\cle_\vp$ if and only if $B^*_{\bar{\vp}}h=0$. Therefore, $\cle_\vp=\cln(B^*_{\bar{\vp}})$.
\end{proof}

The following result gives the coupled Toeplitz--Hankel equations
associated with the equation $D_\vp D_\vp^*h=h$.

\begin{proposition}\label{prop:coupled-TH}
Let $\vp\in L^\infty$ satisfy $|\vp|=1$ a.e.\ on $\T$, and let $g_{+}\in H^2$ and $f_{-}\in H^2_{-}$. Set
\[
h=ug_{+}+f_{-}\in\clk_u^\perp.
\]
Then the following are equivalent:
\begin{enumerate}
\item $D_\vp D_\vp^*h=h$;
		
\item $P_{\clk_u}(\bar{\vp} h)=0$;
		
\item $T_{u\bar{\vp}}g_{+} + H_\vp^*f_{-}\in uH^2$;

\item $H_{\bar u}\left(T_{u\bar{\vp}}g_{+} + H_\vp^*f_{-}\right)=0$.
\end{enumerate}
\end{proposition}

\begin{proof}
As before, using \eqref{DTTO-THO-relation}
\[
I-D_\vp D_\vp^*=B_\vp B_\vp^*,
\]
we obtain
\[
D_\vp D_\vp^*h=h \iff B_\vp^*h=0\iff P_{\clk_u}(\bar{\vp} h)=0.
\]
Thus \textup{(i)} and \textup{(ii)} are equivalent.

For every $F\in L^2$, since $P_{\clk_u}F=P_{\clk_u}P_{+}F$ and
$H^2=\clk_u\oplus uH^2$, we have
\[
P_{\clk_u}F=0
\quad\Longleftrightarrow\quad
P_{+}F\in uH^2.
\]
Applying this to $F=\bar{\vp}h=\bar{\vp}(ug_{+}+f_{-})$, we obtain
\[
P_{\clk_u}(\bar{\vp}h)=0
\quad\Longleftrightarrow\quad
P_{+}\left(\bar{\vp}(ug_{+}+f_{-})\right)\in uH^2.
\]
Moreover,
\[
P_{+}(\bar{\vp}ug_{+})=T_{u\bar{\vp}}g_{+}, \qquad P_{+}(\bar{\vp} f_{-})=H_\vp^*f_{-}.
\]
Hence
\[
P_{\clk_u}(\bar{\vp} h)=0 \iff T_{u\bar{\vp}}g_{+} + H_\vp^*f_{-}\in uH^2,
\]
which proves \textup{(ii)}$\iff$\textup{(iii)}.
	
Finally, for $F\in H^2$,
\[
F\in uH^2\iff H_{\bar u}F=0.
\]
Therefore, \textup{(iii)} and \textup{(iv)} are equivalent.
\end{proof}

Consider the orthogonal decomposition $L^2=\clk_u \oplus \clk_u^\perp$. Then for any  $f\in L^2$, we can write
\begin{equation}\label{eq:key-relations}
\|f\|^2=\|P_{\clk_u}f\|^2+\|P_{\clk_u^\perp}f\|^2.
\end{equation}
The following lemma provides a componentwise characterization of
extremal vectors for a DTTO with unimodular symbol.

\begin{lemma}\label{lem:proj-id}
Let $\vp\in L^\infty$ with $|\vp|=1$ a.e. on $\T$. Then for $h\in H^2$ and $g\in H^2_{-}$,
\begin{enumerate}
\item $	\|D_\vp(uh)\|=\|uh\| \iff P_{\clk_u}(\vp uh)=0
\iff P_{+}(\vp uh)\in uH^2$;

\item $\|D_\vp g\|=\|g\| \iff P_{\clk_u}(\vp g)=0
\iff P_{+}(\vp g)\in uH^2$.
\end{enumerate}
\end{lemma}

\begin{proof}
By Proposition~\ref{prop:extremal-set}, since $|\vp|=1$ a.e. on $\T$,
for every $f\in\clk_u^\perp$ we have
\[
\|D_\vp f\|=\|f\| \quad\Longleftrightarrow\quad
M_\vp f\in\clk_u^\perp.
\]
Equivalently,
\[
\|D_\vp f\|=\|f\|
\quad\Longleftrightarrow\quad
P_{\clk_u}(\vp f)=0.
\]
Taking $f=uh\in uH^2$ and $f=g\in H^2_{-}$ gives the first equivalence in \textup{(i)} and \textup{(ii)}, respectively.
	
It remains to observe that, for every $F\in L^2$,
\[
P_{\clk_u}F=P_{+}F-uP_{+}(\bar{u} F).
\]
Hence
\[
P_{\clk_u}F=0
\quad\Longleftrightarrow\quad
P_{+}F=uP_{+}(\bar{u} F)\in uH^2.
\]
Applying this with $F=\vp uh$ and $F=\vp g$ proves the result.
\end{proof}

\begin{proposition}\label{prop:HT-kernel-componentwise}
Let $\vp\in L^\infty$ satisfy $|\vp|=1$ a.e.\ on $\T$. Then
\[
\cle_\vp^{(+)}=u\cln\left(H_{\bar{u}}T_{u\vp}\right),\qquad
\cle_\vp^{(-)}=V\cln\left(H_{\bar{u}}T_{u\bar{\vp}}\right).
\]
Consequently,
\[
D_\vp \text{ attains its norm on }uH^2\iff \cln\left(H_{\bar{u}}T_{u\vp}\right)\neq\{0\},
\]
and
\[
D_\vp \text{ attains its norm on }H^2_{-}\iff
\cln\left(H_{\bar{u}}T_{u\bar{\vp}}\right)\neq\{0\}.
\]
In particular, if either of these kernels is nonzero, then $D_\vp$ is norm attaining on $\clk_u^\perp$.
\end{proposition}

\begin{proof}
Let $h\in H^2$. By Lemma~\ref{lem:proj-id}, $uh\in\cle_\vp^{(+)}$ if and only if $P_{+}(\vp uh)\in uH^2$. Since $P_{+}(\vp uh)=T_{u\vp}h$, this is equivalent to $T_{u\vp}h\in uH^2$.
	
For $F\in H^2$, we have
\[
F\in uH^2 \iff \bar{u}F\in H^2 \iff H_{\bar{u}}F=0.
\]
Therefore, $uh\in\cle_\vp^{(+)}$ if and only if $H_{\bar{u}}T_{u\vp}h=0$, and hence
\[
\cle_\vp^{(+)}=u\cln\left(H_{\bar{u}}T_{u\vp}\right).
	\]
	
For the coanalytic component, let $g\in H^2_{-}$. Since $D_{\bar{\vp}}=D_\vp^*=C_uD_\vp C_u$ and $C_u(H^2_{-})=uH^2$, the vector $g$ is extremal for $D_\vp$ if and only if $C_ug$ is an analytic extremal for $D_{\bar{\vp}}$. Moreover, $C_ug=uVg$. Applying the first part with $\bar{\vp}$ in place of $\vp$, we obtain
\[
g\in\cle_\vp^{(-)}\iff Vg\in\cln\left(H_{\bar{u}}T_{u\bar{\vp}}\right).
\]
Since $V=V^{-1}$, it follows that $\cle_\vp^{(-)} =V\cln\left(H_{\bar{u}}T_{u\bar{\vp}}\right)$.
	
The two componentwise norm-attainment criteria follow immediately.
In particular, a nonzero vector in either kernel produces a nonzero
extremal vector for $D_\vp$, and hence $D_\vp$ is norm attaining.
\end{proof}

\begin{remark}
The converse of the final consequence need not hold. Indeed, $D_\vp$ may be norm attaining because of a mixed extremal vector $x\oplus y$, while $\cle_\vp^{(+)}=\cle_\vp^{(-)}=\{0\}$. In that case,
\[
\cln\left(H_{\bar{u}}T_{u\vp}\right)=\cln \left(H_{\bar{u}}T_{u\bar{\vp}}\right)=\{0\}.
\]
Thus the proposition characterizes componentwise norm attainment, but not global norm attainment of $D_\vp$.
\end{remark}

A central difficulty in understanding extremal vectors for $D_\vp$
with respect to the orthogonal decomposition
\[
\clk_u^\perp=uH^2\oplus H^2_{-}
\]
is that the orthogonality of the two components is generally not
preserved by $D_\vp$. Indeed, although $x\in uH^2$ and
$y\in H^2_{-}$ satisfy $x\perp y$, there is no a priori reason for
$D_\vp x$ and $D_\vp y$ to be orthogonal. Thus, for an extremal vector $f=x\oplus y$, the norm behaviour of the two components cannot in general be studied independently: any loss of norm in one component may be compensated by the interaction between their images.

The preceding results reduce extremality, for unimodular symbols, to
the condition that multiplication by $\vp$ preserves $\clk_u^\perp$. The next lemma exploits this characterization to show that the interaction between the two components is in fact completely
rigid. More precisely, extremality forces the two componentwise norm
defects to coincide and determines the mixed inner product exactly:
\[
\|D_\vp x\|^2-\|x\|^2=\|D_\vp y\|^2-\|y\|^2=-\langle D_\vp x,D_\vp y\rangle.
\]
In particular, the mixed inner product is necessarily real and
nonnegative. As a consequence, for
\[
F(t):=\|D_\vp(tx+y)\|^2-\|tx+y\|^2,
\qquad t\in\R,
\]
the apparent quadratic interaction collapses to a single perfect
square. This identity makes precise how the analytic and coanalytic
components compensate one another in a mixed extremal vector and will
be used to obtain the phase-rotation criterion below.

\begin{lemma}\label{lem:quad-identities}
Let $u$ be inner and let $\vp\in L^\infty$ satisfy $|\vp|=1$ a.e. on $\T$. Suppose that $0\neq f=x\oplus y\in\clk_u^\perp$ is extremal for $D_\vp$, where $x\in uH^2$, and $y\in H^2_{-}$. For $t\in\R$, define
\[
F(t):=\|D_\vp(tx+y)\|^2-\|tx+y\|^2.
\]
Set
\[
\alpha:=\|D_\vp x\|^2-\|x\|^2, \quad\text{and}\quad \gamma:=\|D_\vp y\|^2-\|y\|^2.
\]
Then $\alpha=\gamma\leq 0$ and $\langle D_\vp x,D_\vp y\rangle=-\alpha\geq0$. In particular,
\[
\|D_\vp x\|^2-\|x\|^2=\|D_\vp y\|^2-\|y\|^2 =-\langle D_\vp x,D_\vp y\rangle,
\]
and $F(t)=\alpha(t-1)^2\leq0$, for $t\in \R$.
\end{lemma}

\begin{proof}
Since $f=x\oplus y$ is extremal and $|\vp|=1$ a.e. on $\T$, Proposition~\ref{prop:extremal-set} gives
\[
M_\vp f=M_\vp(x+y)\in\clk_u^\perp.
\]
Set
\[
p:=P_{\clk_u}(\vp x),\qquad q:=P_{\clk_u}(\vp y).
\]
Then $p+q=P_{\clk_u}\left(\vp(x+y)\right) =0$, and hence
\[
q=-p.
\]
Since $D_\vp x=P_{\clk_u^\perp}(\vp x)$, and $D_\vp y=P_{\clk_u^\perp}(\vp y)$, we have the orthogonal decompositions
\[
\vp x=D_\vp x+p,\qquad \vp y=D_\vp y+q.
\]
Because $x\perp y$ and $M_\vp$ is unitary, we have
\[
0=\langle \vp x,\vp y\rangle.
\]
Using the orthogonality of $\clk_u$ and $\clk_u^\perp$, we obtain
\[
0=\langle D_\vp x,D_\vp y\rangle+\langle p,q\rangle =\langle D_\vp x,D_\vp y\rangle-\|p\|^2.
\]
Therefore,
\[
\langle D_\vp x,D_\vp y\rangle=\|p\|^2\geq0.
\]
On the other hand, since $|\vp|=1$ a.e. on $\T$,
\[
\|x\|^2=\|\vp x\|^2=\|D_\vp x\|^2+\|p\|^2.
\]
Thus
\[
\alpha=\|D_\vp x\|^2-\|x\|^2=-\|p\|^2.
\]
Similarly, since $q=-p$,
\[
\gamma=\|D_\vp y\|^2-\|y\|^2=-\|q\|^2=-\|p\|^2=\alpha.
\]
Consequently,
\[
\langle D_\vp x,D_\vp y\rangle=-\alpha\qquad\text{and}\qquad \alpha=\gamma\leq0.
\]
Finally,
\[
\begin{aligned}
F(t)&=\alpha t^2+2\Re\langle D_\vp x,D_\vp y\rangle t+\gamma\\
	&=\alpha t^2-2\alpha t+\alpha\\
	&=\alpha(t-1)^2.
\end{aligned}
\]
This proves the result.
\end{proof}

\begin{example}
Let $u(z)=z$ for $z\in\T$, and let $\vp(z)= z^{k}$ for some integer $k$ with $|k|\ge 2$. Then $|\vp|=1$ a.e on $\T$. Since $\clk_u^\perp = zH^2 \oplus H^2_{-}$, choose
\[
x=z \in zH^2, \qquad y=\bar{z}=z^{-1} \in H^2_{-},
\]
and set $f:=x\oplus y$, which we identify in $L^2$ with $x+y=z+z^{-1}$. Then
\[
\|f\|^2 = \|z\|^2 + \|z^{-1}\|^2 = 1+1 = 2.
\]
We have
\[
M_\vp f = z^k(z+z^{-1}) = z^{k+1} + z^{k-1}.
\]
Thus
\[
D_\vp f = P_{\clk_u^\perp}(M_\vp f).
\]
\textit{Case $k\ge 2$.}
Here $k+1\ge 3$ and $k-1\ge 1$, so
\[
z^{k+1}, z^{k-1} \in zH^2 \subset \clk_u^\perp.
\]
Hence
\[
D_\vp x = P_{\clk_u^\perp}(z^{k+1}) = z^{k+1}, \qquad
D_\vp y = P_{\clk_u^\perp}(z^{k-1}) = z^{k-1},
\]
so that
\[
D_\vp f = D_\vp x + D_\vp y = z^{k+1} + z^{k-1}.
\]
Since $z^{k+1}$ and $z^{k-1}$ are orthogonal in $L^2$, we obtain
\[
\|D_\vp f\|^2 = \|z^{k+1}\|^2 + \|z^{k-1}\|^2 = 1+1 = 2 = \|f\|^2,
\]
and
\[
\langle D_\vp x, D_\vp y\rangle = \langle z^{k+1}, z^{k-1}\rangle
= 0.
\]
Thus $f=x\oplus y$ is a norm attaining vector for $D_\vp$, and at the same time the mixed term $\langle D_\vp x, D_\vp y\rangle$ vanishes.

\textit{Case $k\le -2$.}
In this case $k+1\le -1$ and $k-1\le -3$, so $z^{k+1},z^{k-1}\in H^2_{-}\subset\clk_u^\perp$.
The same computation gives
\[D_\vp x = z^{k+1},\quad D_\vp y = z^{k-1},\quad
\|D_\vp f\|^2 = 2 = \|f\|^2,\qquad
\langle D_\vp x, D_\vp y\rangle = 0.\]
Again $f$ is norm attaining and the mixed term vanishes.
	
	\smallskip
In particular, for any integer $k$ with $|k|\ge 2$, the vector $f=x\oplus y$ is extremal for $D_\vp$ and satisfies
\[\langle D_\vp x, D_\vp y\rangle = 0,\]
so this furnishes a nontrivial family of examples with vanishing mixed term.
\end{example}

\begin{example}
Let $u$ be a nonconstant inner function so that
\[
\clk_u^\perp \;=\; uH^2 \oplus H^2_{-} \subset L^2.
\]
Define a unimodular symbol
\[
\vp(e^{it}) := \overline{u(e^{it})}, \qquad t\in [0,2\pi).
\]
	
Now choose
\[
x = u^2 \in uH^2, \qquad y= \bar{z} = z^{-1} \in H^2_{-},
\]
and set $f := x \oplus y \in \clk_u^\perp$. Identifying $f$ with $x+y$ in $L^2$, we have
\[
\|f\|^2 = \|u^2\|^2 + \|\bar{z}\|^2 = 1 + 1 = 2.
\]
	
Now compute
\[
M_\vp f = \vp(x+y) = \bar{u}\,(u^2 + \bar{z})= u + \bar{u}\,\bar{z}.
\]
Since $u \in uH^2$ and $\bar{u}$ is anti-analytic, we have $\bar{u}\,\bar{z} \in H^2_{-}$. Thus
\[
M_\vp f \in uH^2 \oplus H^2_{-} = \clk_u^\perp,
\]
and hence
\[
D_\vp f = P_{\clk_u^\perp}(M_\vp f) = M_\vp f = u + \bar{u}\,\bar{z}.
\]
In particular,
\[
D_\vp x = u \in uH^2, \qquad D_\vp y = \bar{u}\,\bar{z} \in H^2_{-}.
\]
	
Since $uH^2 \perp H^2_{-}$ in $L^2$, we have
\[
\langle D_\vp x, D_\vp y\rangle= \langle u, \bar{u}\,\bar{z}\rangle = 0,
\]
so the mixed term vanishes. Moreover,
\[
\|D_\vp f\|^2= \|u\|^2 + \|\bar{u}\,\bar{z}\|^2= 2 = \|f\|^2.
	\]
Thus $f=x\oplus y$ is a norm-attaining vector for $D_\vp$, and at the same time the mixed term $\langle D_\vp x, D_\vp y\rangle$ vanishes. Therefore, for every nonconstant inner function $u$, this construction produces a unimodular symbol $\vp\in L^\infty$ and a nonzero $f\in\clk_u^\perp$ such that
\[
\|D_\vp f\| = \|f\|\quad\text{and}\quad\langle D_\vp x, D_\vp y\rangle = 0.
\]
\end{example}

\begin{remark}
These examples illustrate the subtlety underlying the mixed term $\langle D_\vp x, D_\vp y\rangle$. Even though $x\perp y$ in $L^2$, the images $D_\vp x$ and $D_\vp y$ interact in a delicate way because multiplication by $\vp$ and the subsequent projection $P_{\clk_u^\perp}$ can mix analytic and coanalytic components in a nontrivial manner, despite the unimodularity of $\vp$. For this reason it is important to examine the extremality of $D_\vp$ separately on the analytic part $uH^2$ and the coanalytic part $H^2_{-}$. Although each component is well understood individually, their interaction under $D_\vp$ can be highly nontrivial. Lemma~\ref{lem:quad-identities} shows that whenever $x\oplus y$ is extremal, these interactions are not arbitrary but must satisfy the rigid balance
\[
\|D_\vp x\|^2-\|x\|^2=\|D_\vp y\|^2-\|y\|^2=-\langle D_\vp x, D_\vp y\rangle,
\]
forcing the associated quadratic to collapse to a perfect square. This structural rigidity explains precisely when the mixed term can vanish and how such vanishing fits naturally into the extremality behaviour of $D_\vp$ on the analytic and coanalytic components.
\end{remark}

\begin{corollary}\label{cor:zeroMixedTerm}
Let $u$ be inner and let $\vp\in L^\infty$ satisfy $|\vp|=1$ a.e.\ on $\T$. Suppose that $0\neq f=x\oplus y\in\clk_u^\perp$ is extremal for $D_\vp$, where $x\in uH^2$ and $y\in H^2_{-}$. For $\theta\in\R$, set
\[
f_\theta:=x\oplus e^{i\theta}y, \qquad
c:=\langle D_\vp x,D_\vp y\rangle.
\]
Then $c\geq0$, and the following hold:
\begin{enumerate}
\item $f_\theta$ is extremal for $D_\vp$ if and only if $c\cos\theta=c$.
		
\item If $c=0$, then $f_\theta$ is extremal for every $\theta\in\R$.
		
\item If $c>0$, then $f_\theta$ is extremal if and only if $\theta\in2\pi\Z$.
		
\item The following are equivalent:
\[
c=0
\quad\Longleftrightarrow\quad\|D_\vp x\|=\|x\|\quad\Longleftrightarrow\quad\|D_\vp y\|=\|y\|.
\]
In particular, if one component is extremal, then so is the other.
\end{enumerate}
\end{corollary}

\begin{proof}
By Lemma~\ref{lem:quad-identities},
\[
\|D_\vp x\|^2-\|x\|^2=\|D_\vp y\|^2-\|y\|^2=-c,\qquad c\geq0.
\]
	
For $\theta\in\R$, using $x\perp y$, we have
\[
\begin{aligned}
\|D_\vp f_\theta\|^2-\|f_\theta\|^2
&=\left(\|D_\vp x\|^2-\|x\|^2\right)+\left(\|D_\vp y\|^2-\|y\|^2\right) \\
&\quad
+2\Re\left(e^{-i\theta}\langle D_\vp x,D_\vp y\rangle\right) \\
&=-2c+2c\cos\theta \\
&=2c(\cos\theta-1).
\end{aligned}
\]
Hence $f_\theta$ is extremal if and only if $c(\cos\theta-1)=0$, proving \textup{(i)}.
	
If $c=0$, the above identity vanishes for every $\theta$, proving
\textup{(ii)}. 

If $c>0$, extremality is equivalent to $\cos\theta=1$, that is, $\theta\in2\pi\Z$, proving \textup{(iii)}.
	
Finally, again from
\[
\|D_\vp x\|^2-\|x\|^2=\|D_\vp y\|^2-\|y\|^2=-c,
\]
we have $c=0$ if and only if $\|D_\vp x\|=\|x\|$, and this is equivalent to $\|D_\vp y\|=\|y\|$. This proves \textup{(iv)}.
\end{proof}

For the rest of this subsection, let $u$ be a nonconstant inner function and let $\vp\in L^\infty$ satisfy $|\vp|=1$ a.e. on $\T$. Define
\[
T_{+}:H^2\longrightarrow\clk_u
\quad\text{by}\quad
T_{+}(h):=P_{\clk_u}(\vp uh),
\]
and
\[
T_{-}:H^2_{-}\longrightarrow\clk_u
\quad\text{by}\quad
T_{-}(g):=P_{\clk_u}(\vp g).
\]
Set
\[
M_{+}:=\cln(T_{+}),
\qquad
M_{-}:=\cln(T_{-}).
\]

\begin{proposition}\label{prop:Tpm-extremal-NA}
Let $f=uh\oplus g\in\clk_u^\perp$, where $h\in H^2$ and $g\in H^2_{-}$. Then the following hold:
\begin{enumerate}
\item $f$ is extremal for $D_\vp$ if and only if
		\[
		T_{+}h+T_{-}g=0.
		\]
		
\item $D_\vp$ is norm attaining if and only if at least one of the following conditions holds:
\begin{enumerate}
\item[\textup{(a)}] $M_{+}\neq\{0\}$;
\item[\textup{(b)}] $M_-\neq\{0\}$;
\item[\textup{(c)}] $\clr(T_{+})\cap\clr(T_{-})\neq\{0\}$.
\end{enumerate}
\end{enumerate}
\end{proposition}

\begin{proof}
Since $|\vp|=1$ a.e. on $\T$, Proposition~\ref{prop:extremal-set}
gives
\[
f\in\cle_\vp \iff P_{\clk_u}(\vp f)=0.
\]
For $f=uh\oplus g$, the definitions of $T_{+}$ and $T_{-}$ yield
\[
P_{\clk_u}(\vp f)=T_{+}h+T_{-}g.
\]
Hence
\[
f\in\cle_\vp \iff T_{+}h+T_{-}g=0,
\]
which proves \textup{(i)}.
	
To prove \textup{(ii)}, suppose first that $D_\vp$ is norm attaining. Then there exists a nonzero extremal vector $f=uh\oplus g$. By \textup{(i)},
\[
T_{+}h+T_{-}g=0.
\]
If $g=0$, then $0\neq h\in M_{+}$, so \textup{(a)} holds. If $h=0$,
then $0\neq g\in M_-$, so \textup{(b)} holds.
	
Suppose that both $h$ and $g$ are nonzero. If $T_{+}h=0$, then also
$T_{-}g=0$, and hence both \textup{(a)} and \textup{(b)} hold. Otherwise,
\[
0\neq T_{+}h=-T_{-}g\in\clr(T_{+})\cap\clr(T_{-}),
\]
so \textup{(c)} holds.
	
Conversely, if \textup{(a)} holds, choose $0\neq h\in M_{+}$. Then
$T_{+}h+T_{-}0=0$, so $uh\oplus0$ is a nonzero extremal vector by
\textup{(i)}. Similarly, \textup{(b)} produces a nonzero extremal
vector of the form $0\oplus g$.
	
Finally, suppose that \textup{(c)} holds. Choose $0\neq w\in\clr(T_{+})\cap\clr(T_{-})$. Then there exist $h\in H^2$ and $g\in H^2_{-}$ such that
\[
T_{+}h=w,\qquad T_{-}g=w.
\]
Therefore,
\[
T_{+}h+T_{-}(-g)=0.
\]
Since $w\neq0$, the vector $uh\oplus(-g)$ is nonzero, and by \textup{(i)} it is extremal for $D_\vp$. Thus $D_\vp$ is norm attaining.
\end{proof}

\begin{corollary}\label{cor:outer-extremal-range}
Suppose that $0\neq h_0\in M_{+}$ and let $h_0=\psi_{+}k_0$ be its inner--outer factorization, where $\psi_{+}$ is inner and $k_0$ is outer. Set
\[
F:=\vp u\psi_{+},\qquad T_Fh:=P_{+}(Fh),\quad h\in H^2.
\]
Then $T_F(k_0)\in uH^2$, and $T_F(H^2)\subseteq\cls$, where
\[
\cls=\overline{uH^2+\operatorname{span}\left\{P_{+}\left(z^N(Fk_0)_-\right):N\geq0\right\}}.
\]
\end{corollary}

\begin{proof}
Since $h_0\in M_{+}=\cln (T_{+})$,
\[
P_{\clk_u}(\vp uh_0)=0.
\]
By Lemma~\ref{lem:proj-id}, we have $P_{+}(\vp uh_0)\in uH^2$. As $h_0=\psi_{+}k_0$ and $F=\vp u\psi_{+}$, it follows that
\[
T_F(k_0)=P_{+}(Fk_0)=P_{+}(\vp uh_0)\in uH^2.
\]
Write
\[
Fk_0=(Fk_0)_{+}+(Fk_0)_{-},
\]
with $(Fk_0)_{+}\in H^2$ and $(Fk_0)_{-}\in H^2_{-}$. For every $N\geq0$,
\[
T_F(k_0z^N)=P_{+}(z^NFk_0)=z^N(Fk_0)_{+}+P_{+}\left(z^N(Fk_0)_-\right).
\]
Since $(Fk_0)_{+}=T_F(k_0)\in uH^2$ and $uH^2$ is shift invariant, we have
\[
z^N(Fk_0)_+\in uH^2.
\]
Hence
\[
T_F(k_0z^N)\in\cls\qquad(N\geq0).
\]
By linearity, $T_F(k_0p)\in\cls$ for every analytic polynomial $p$.
	
Since $k_0$ is outer, $\{k_0p:p\ \text{is an analytic polynomial}\}$
is dense in $H^2$. As $F\in L^\infty$, the Toeplitz operator $T_F$ is bounded, and $\cls$ is closed by definition. Therefore, $T_F(H^2)\subseteq\cls$.
\end{proof}

\subsection{Norm attainment on the analytic component $uH^2$}
We begin with a fundamental lemma, which is crucial for the proof of the main result of this section.
\begin{lemma}\label{lem:outer-test}
Let $F\in L^\infty$ satisfy $|F|=1$ a.e. on $\T$, and let $k_0\in H^2$
be a nonzero outer function. If
\[
F k_0 \in H^2,
\]
then $F$ is inner.
\end{lemma}

\begin{proof}
Let $G = Fk_0\in H^2$. Then $G$ admits the inner-outer factorization
$G = I O$, where $I$ is inner and $O$ is outer. Since $|F|=1$ a.e. on $\T$, we get
\[|G| = |F k_0| = |k_0| \quad \text{a.e. on } \T.\]  
But also $|G| = |IO| = |O|$ a.e. on $\T$, so
\[
|O| = |k_0|\quad\text{a.e. on }\T.
\]

By \cite[Corollary 6.23]{Douglas1998}, it follows that $O = c\,k_0$ for some $c\in\T$.  
Therefore,
\[
G = IO = I(c k_0) = c I k_0.
\]
Since $G = F k_0$ and $k_0\neq0$ a.e. on $\T$, the identity
$Fk_0=cIk_0$ implies $F=cI$ a.e. on $\T$ which is inner.
\end{proof}

\subsection{A factorization criterion for analytic extremals}

We next obtain a factorization criterion for analytic extremal
vectors satisfying an additional Hardy-space condition. We emphasize
that the condition
\[
\vp u h_0\in H^2
\]
below is an additional hypothesis and does not follow merely from
the extremality of $uh_0$.

\begin{theorem}\label{thm:analytic-Dphi}
Let $u$ be a nonconstant inner function and let $\vp\in L^\infty$ satisfy $\|\vp\|_\infty=1$. Then the following are equivalent:
\begin{enumerate}
\item\label{analytic-Dphi-1}
There exist inner functions $\psi_{+}$ and $\chi_{+}$ such that
\[
\vp=\bar{u}\,\bar{\psi}_{+}\chi_{+}\quad \text{a.e. on }\T.
\]
		
\item\label{analytic-Dphi-2}
There exists a nonzero $h_0\in H^2$ such that
\[
f_0:=uh_0\oplus0\in\clk_u^\perp
\]
is extremal for $D_\vp$, that is,
\[
\|D_\vp f_0\|=\|f_0\|,
\]
and, in addition,
\[
\vp uh_0\in H^2.
\]
\end{enumerate}
\end{theorem}

\begin{proof}
\textbf{(\ref{analytic-Dphi-1}) $\Rightarrow$ (\ref{analytic-Dphi-2}).}
Suppose that $\vp=\bar{u}\,\bar{\psi}_{+}\chi_{+}$, where $\psi_{+}$ and $\chi_{+}$ are inner. Then
\[
\vp u\psi_{+}=\chi_{+}.
\]
Let $d=\gcd(u,\chi_{+})$ and write
\[
u=du_1,\qquad\chi_{+}=d\chi_1,\qquad\gcd(u_1,\chi_1)=1.
\]
Choose any nonzero $h\in H^2$ and set
\[
h_0:=\psi_{+}u_1h.
\]
Then
\[
\vp uh_0 =\vp u\psi_{+}u_1h=\chi_{+}u_1h=d\chi_1u_1h=u\chi_1h
\in uH^2\subset H^2.
\]
Consequently,
\[
D_\vp(uh_0)=P_{\clk_u^\perp}(\vp uh_0)=\vp uh_0.
\]
Since $u,\psi_{+},u_1,\chi_{+}$ are inner,
\[
\|uh_0\|=\|\psi_{+}u_1h\|=\|h\|,
	\]
and
\[
\|D_\vp(uh_0)\|=\|\chi_{+}u_1h\|=\|h\|.
\]
Thus $uh_0\oplus0$ is extremal for $D_\vp$, and $\vp uh_0\in H^2$.
	
	\medskip
	
\textbf{(\ref{analytic-Dphi-2}) $\Rightarrow$
		(\ref{analytic-Dphi-1}).}
Suppose that there exists $0\neq h_0\in H^2$ such that
\[
\|D_\vp(uh_0)\|=\|uh_0\|
\]
and
\[
\vp uh_0\in H^2.
\]
By Proposition~\ref{prop:componentwise-unimodular}, we have
\[
|\vp|=1\qquad\text{a.e. on }\T.
\]
Write the inner--outer factorization $h_0=\psi_{+}k_0$,
where $\psi_{+}$ is inner and $k_0$ is outer, and set
\[
F:=\vp u\psi_{+}.
\]
Since $|\vp|=|u|=|\psi_{+}|=1$ a.e. on $\T$,
\[
|F|=1\qquad\text{a.e. on }\T.
\]
Moreover,
\[
Fk_0=\vp u\psi_{+}k_0=\vp uh_0\in H^2.
\]
By Lemma~\ref{lem:outer-test}, $F$ is inner. Hence there exists an inner function $\chi_{+}$ such that $F=\chi_{+}$. Therefore,
\[
\vp u\psi_{+}=\chi_{+},
\]
or equivalently,
\[
\vp=\bar{u}\,\bar{\psi}_{+}\chi_{+}.
\]
\end{proof}

\begin{remark}
The additional condition $\vp uh_0\in H^2$ in Theorem~\ref{thm:analytic-Dphi} is not a consequence of extremality alone. Indeed, if $uh_0$ is extremal, then	Proposition~\ref{prop:componentwise-unimodular} and Proposition~\ref{prop:extremal-set} imply only that
\[\vp uh_0\in\clk_u^\perp=uH^2\oplus H^2_{-}.
\]
Thus Theorem~\ref{thm:analytic-Dphi} is a factorization criterion for analytic extremals satisfying the additional Hardy-space condition, rather than a characterization of all
analytic norm attainment.
\end{remark}

 \begin{remark}
 Assume that \textup{(\ref{analytic-Dphi-1})} holds, so that
 \[
 \vp=\bar{u}\,\bar\psi_{+}\chi_{+},
 \]
 where $\psi_{+}$ and $\chi_{+}$ are inner. For $h\in H^2$,
 Lemma~\ref{lem:proj-id} gives
\[
uh\oplus0\in\cle_\vp \quad\Longleftrightarrow\quad
 P_{+}(\vp uh)\in uH^2.
\]
Since $\vp uh=\bar\psi_{+}\chi_{+}h$, the analytic extremal set is
\[
 \cle_\vp^{(+)}=\left\{uh\oplus0:h\in H^2,\;
 P_{+}\left(\bar\psi_{+}\chi_{+}h\right)\in uH^2\right\}.
 \]
 A more explicit description is available for extremal vectors lying
 in $u\psi_{+}H^2$. Let $k\in H^2$ and set
 \[
f_0=u\psi_{+}k\oplus0.
\]
Then $ \vp u\psi_{+}k=\chi_{+}k\in H^2$, and therefore, by Lemma~\ref{lem:proj-id},
\[
f_0\in\cle_\vp \quad\Longleftrightarrow\quad\chi_{+}k\in uH^2.
 \]
Let $d=\gcd(u,\chi_{+})$. Then $u=du_1$, and $\chi_{+}=d\chi_1$,
where $\gcd(u_1,\chi_1)=1$. Then
 \[
 \chi_{+}k\in uH^2
 	\quad\Longleftrightarrow\quad
 	\chi_1k\in u_1H^2.
\]
Since $u_1$ and $\chi_1$ are relatively prime inner functions,
\[
\chi_1k\in u_1H^2 \quad\Longleftrightarrow\quad
k\in u_1H^2.
\]
Consequently,
 \[
 \cle_\vp^{(+)}\cap \left(u\psi_{+}H^2\oplus\{0\}\right)
 =\left\{u\psi_{+}u_1h\oplus0:h\in H^2
 \right\}.
 \]
 \end{remark}

\subsection{A factorization criterion for coanalytic extremals}

We now obtain the coanalytic analogue of
Theorem~\ref{thm:analytic-Dphi}. As in the analytic case, the
additional condition appearing below is not a consequence of
extremality alone.

\begin{theorem}\label{thm:coanalytic-Dphi}
Let $u$ be a nonconstant inner function and let $\vp\in L^\infty$ satisfy $\|\vp\|_\infty=1$. Then the following are equivalent:
\begin{enumerate}
\item\label{coanalytic-Cu-1}
There exist inner functions $\psi_{-}$ and $\chi_{-}$ such that
\begin{equation}\label{eq:coanalytic-factor-phi}
\vp=u\psi_{-}\overline{\chi_{-}}.
\end{equation}
		
\item\label{coanalytic-Cu-2}
There exists a nonzero $y_0\in H^2_{-}$ such that $y_0$ is extremal for $D_\vp$, that is, $\|D_\vp y_0\|=\|y_0\|$, and
\begin{equation}\label{eq:extracondition-2}
\vp\bar{u}\,y_0\in H^2_{-}.
\end{equation}
\end{enumerate}
\end{theorem}

\begin{proof}
\textbf{\textup{(\ref{coanalytic-Cu-1})}
		$\Rightarrow$\textup{(\ref{coanalytic-Cu-2})}.}
Suppose that$\vp=u\psi_{-}\overline{\chi_{-}}$ for inner functions
$\psi_{-}$ and $\chi_{-}$. Then
\[
\bar{\vp}=\bar{u}\,\overline{\psi_{-}}\chi_{-}.
\]
Applying Theorem~\ref{thm:analytic-Dphi} to $\bar{\vp}$, there exists
a nonzero $h_0\in H^2$ such that $uh_0$ is extremal for $D_{\bar{\vp}}$ and $\bar{\vp}\,uh_0\in H^2$.
	
Set
\[
y_0:=C_u(uh_0).
\]
Since $C_u(uH^2)=H^2_{-}$, we have $0\neq y_0\in H^2_{-}$. Moreover,
using $D_{\bar{\vp}}=D_\vp^*=C_uD_\vp C_u$ and $C_u^2=I$, we obtain
\[
\|D_\vp y_0\|=\|C_uD_\vp C_u(uh_0)\|
	=\|D_{\bar{\vp}}(uh_0)\|=\|uh_0\|=\|y_0\|.
\]
Thus $y_0$ is extremal for $D_\vp$.
	
It remains to verify \eqref{eq:extracondition-2}. Since $y_0=C_u(uh_0)=\bar{z}\overline{h_0}=Vh_0$, we have
\[
V(\bar{\vp}\,uh_0)=\bar{z}\vp\bar{ u}\overline{h_0}
=\vp\bar{ u}y_0.
	\]
As $\bar{\vp}uh_0\in H^2$ and $V(H^2)=H^2_{-}$, it follows that $\vp\bar{u}\,y_0\in H^2_{-}$. Hence \textup{(\ref{coanalytic-Cu-2})} holds.
	
\medskip
	
\textbf{\textup{(\ref{coanalytic-Cu-2})} $\Rightarrow$ \textup{(\ref{coanalytic-Cu-1})}.}
Suppose that there exists $0\neq y_0\in H^2_{-}$ such that $\|D_\vp y_0\|=\|y_0\|$ and $\vp\bar{u}\,y_0\in H^2_{-}$.
	
Since $C_u(H^2_{-})=uH^2$, there exists a nonzero $h_0\in H^2$ such that
\[
C_uy_0=uh_0.
\]
Using $D_{\bar{\vp}}=C_uD_\vp C_u$, we obtain
\[
\|D_{\bar{\vp}}(uh_0)\|=\|C_uD_\vp y_0\|=\|D_\vp y_0\|=\|y_0\|=\|uh_0\|.
	\]
Hence $uh_0$ is an analytic extremal for $D_{\bar{\vp}}$.
	
Since $y_0=C_u(uh_0)=Vh_0$, the same computation as above gives
\[
V(\bar{\vp}\,uh_0)=\vp\bar{u}\,y_0.
\]
By assumption, the right-hand side belongs to $H^2_{-}$. Since
$V(H^2_{-})=H^2$ and $V^2=I$, it follows that $\bar{\vp}\,uh_0\in H^2$.
	
We may therefore apply Theorem~\ref{thm:analytic-Dphi} to the symbol
$\bar{\vp}$. Hence there exist inner functions $\psi_{-}$ and $\chi_{-}$
such that
\[
\bar{\vp}=\bar{u}\,\overline{\psi_{-}}\chi_{-}.
\]
Taking complex conjugates yields $\vp=u\psi_{-}\overline{\chi_{-}}$, which is \eqref{eq:coanalytic-factor-phi}.
\end{proof}

\begin{remark}
The conjugation $C_u$ gives an exact correspondence between coanalytic extremals of $D_\vp$ and analytic extremals of $D_{\bar{\vp}}$:
\[
\cle_\vp^{(-)}=C_u\left(\cle_{\bar{\vp}}^{(+)}\right).
\]
Thus, under \eqref{eq:coanalytic-factor-phi}, the description of the analytic extremal set for $D_{\bar{\vp}}$ gives
\[
\cle_\vp^{(-)}=\left\{0\oplus Vh:h\in H^2,\;P_{+}\left(\bar\psi_{-}\chi_{-}h\right)\in uH^2\right\}.
\]
	
A more explicit description is available on the corresponding special subspace. Let $d=\gcd(u,\chi_{-})$ and write $u=du_1$, $\chi_{-}=d\chi_1$, where $\gcd(u_1,\chi_1)=1$. Since
\[
C_u(u\psi_{-}u_1h)=\bar\psi_{-}\bar{u}_1V(h),
\]
we obtain
\[
\cle_\vp^{(-)}\cap\left(\{0\}\oplus\bar\psi_{-}H^2_{-}\right)
=\left\{0\oplus\bar\psi_{-}\bar{u}_1g:g\in H^2_{-}\right\}.
\]
In particular, the last set describes the coanalytic extremals
arising from the special analytic extremals $u\psi_{-}u_1H^2$ of $D_{\bar{\vp}}$; it need not exhaust all coanalytic extremals of $D_\vp$.
\end{remark}

\begin{theorem}\label{thm:quotient-inner-extremals}
Let $u$ be a nonconstant inner function and let $\vp=r\overline{a}b$, $r>0$, where $a$ and $b$ are inner functions. Then $D_\vp$ attains its norm
on both components $uH^2$ and $H^2_{-}$. More precisely,
\[
u aH^2\subseteq \cle_\vp^{(+)},\qquad
V(bH^2)\subseteq \cle_\vp^{(-)},
\]
where $Vh=\overline{z}\overline{h}$, $h\in H^2$. In particular, both $\cle_\vp^{(+)}$ and $\cle_\vp^{(-)}$ are infinite-dimensional.
\end{theorem}

\begin{proof}
Since $a$ and $b$ are inner, we have $|\vp|=r$ a.e. on $\T$, and hence $\|\vp\|_\infty=r$.
	
Let $h\in H^2$. For the analytic component, we have
\[
\vp(uah)=r\overline{a}buah
=rubh\in uH^2\subseteq\clk_u^\perp.
\]
Therefore,
\[
D_\vp(uah)=P_{\clk_u^\perp}(\vp uah)=r\,ubh.
\]
Since $u,a$, and $b$ are inner,
\[
\|D_\vp(uah)\|=r\|ubh\|=r\|h\|=r\|uah\|=\|\vp\|_\infty\|uah\|.
\]
Thus every vector in $uaH^2$ is extremal for $D_\vp$, and consequently
\[
uaH^2\subseteq\cle_\vp^{(+)}.
\]
	
For the coanalytic component, let
\[
y=V(bh)=\overline{z}\,\overline{bh}\in H^2_{-}.
	\]
Then
\[
\vp y =r\,\overline{a}\,b\,\overline{z}\,\overline b\,\overline{h}
=r\,\overline{z}\,\overline{a}\,\overline{h}
=rV(ah)\in H^2_{-}\subseteq\clk_u^\perp.
\]
Hence
\[
D_\vp y=rV(ah).
\]
Since $V$ is antiunitary and $a,b$ are inner,
\[
\|D_\vp y\|=r\|V(ah)\|=r\|h\|=r\|V(bh)\|=\|\vp\|_\infty\|y\|.
\]
Therefore every vector in $V(bH^2)$ is extremal, and
\[
V(bH^2)\subseteq\cle_\vp^{(-)}.
\]
	
Finally, multiplication by an inner function and the antiunitary
operator $V$ preserve infinite dimensionality. Hence both
$\cle_\vp^{(+)}$ and $\cle_\vp^{(-)}$ are infinite-dimensional.
\end{proof}

\subsection{Examples}

In this subsection, we illustrate Theorem~\ref{thm:analytic-Dphi} with two fully explicit examples. The first one is genuinely nontrivial: the unimodular symbol $\vp$ is a non-analytic quotient of Blaschke products, and the analytic extremal subspace $\cle_\vp^{(+)}$ involves nontrivial inner factors. The second example is the simplest possible case ($u(z)=z$, $\vp=\bar{z}$), which nevertheless illustrates the mechanism of the theorem very clearly. Subsequently, we discuss a few additional examples of norm attaining DTTOs.

\begin{example}\label{ex:nontrivial}
Choose distinct points $a,b\in\D\setminus\{0\}$ and let
\[
B_a(z)=\frac{z-a}{1-\overline{a} z},\qquad
B_b(z)=\frac{z-b}{1-\overline b z}
\]
be the corresponding Blaschke factors. Set
\[
u(z)=z,\qquad \psi_{+}(z)=B_a(z),\qquad \chi_{+}(z)=zB_b(z),
\]
and define, for $z\in\T$,
\[
\vp(z):=\frac{\chi_{+}(z)}{u(z)\psi_{+}(z)} =\frac{B_b(z)}{B_a(z)}.
\]
Since $B_a$ and $B_b$ are inner, $\vp$ is unimodular on $\T$ and
$\|\vp\|_\infty=1$. Moreover,
\[
\vp u\psi_{+}=\frac{B_b}{B_a}\,zB_a=zB_b=\chi_{+}.
\]
Thus $\vp=\bar{u}\,\bar\psi_{+}\chi_{+}$, so the factorization in Theorem~\ref{thm:analytic-Dphi} holds.
	
Since $\gcd(u,\chi_{+})=\gcd(z,zB_b)=z$, we may write $u=du_1$ and $\chi_{+}=d\chi_1$ with
\[
d=z,\qquad u_1=1,\qquad \chi_1=B_b.
\]
Hence the preceding remark gives the family $zB_aH^2$ of analytic extremal vectors. We now show that, in this example, this is in fact the entire analytic extremal space.
	
Let $h\in H^2$. By Lemma~\ref{lem:proj-id},
\[
zh\oplus0\in\cle_\vp\quad\Longleftrightarrow\quad
P_{+}(\vp zh)\in zH^2.
\]
Since $u=z$, the latter condition is equivalent to the vanishing of
the constant Fourier coefficient of $\vp zh$. Now, using the boundary identity
\[
\overline{B_a(z)} =\frac{1-\overline{a}z}{z-a}, \qquad z\in\T,
\]
we obtain
\[
\begin{aligned}
	\widehat{\vp zh}(0)
	&=\frac{1}{2\pi i}\int_{\T}\vp(z)h(z)\,dz \\
	&=\frac{1}{2\pi i}\int_{\T}
	\frac{B_b(z)(1-\overline{a}z)h(z)}{z-a}\,dz \\
	&=(1-|a|^2)B_b(a)h(a).
\end{aligned}
\]
Indeed, the function
\[
F(z):=B_b(z)(1-\overline{a}z)h(z)
\]
belongs to $H^2$, and the Cauchy integral formula for $H^2$ boundary
functions gives
\[
\frac{1}{2\pi i}\int_{\T}\frac{F(z)}{z-a}\,dz =F(a)=(1-|a|^2)B_b(a)h(a).
\]

Since $a\neq b$, we have $B_b(a)\neq0$, and also $1-|a|^2>0$.
Consequently,
\[
\widehat{\vp zh}(0)=0
\quad\Longleftrightarrow\quad
h(a)=0.
\]
Since $B_a$ is the Blaschke factor with zero at $a$, the standard
factorization property of $H^2$ yields
\[
h(a)=0
\quad\Longleftrightarrow\quad
h\in B_aH^2.
\]
Therefore,
\[
\widehat{\vp zh}(0)=0
\quad\Longleftrightarrow\quad
h(a)=0
\quad\Longleftrightarrow\quad
h\in B_aH^2.
\]
Consequently,
\[
\cle_\vp^{(+)}=\{zh\oplus0:h\in B_aH^2\}=\{zB_a g\oplus0:g\in H^2\}.
\]

and the example gives norm attainment for the generally non-analytic
unimodular symbol $\vp=B_b/B_a$.
\end{example}

\begin{example}\label{ex:trivial}
Let $u(z)=z$ and $\vp(z)=\bar{z}$, $z\in \T$. Then $|\vp|=1$ a.e.\ on $\T$, and $\clk_u=H^2\ominus zH^2=\operatorname{span}\{1\}$, $\clk_u^\perp=zH^2\oplus H^2_{-}$. Since
\[
\vp{u}=\bar{z} z=1 \qquad\text{a.e. on }\T,
\]
we may take
\[
\psi_{+}\equiv1,\qquad\chi_{+}\equiv1.
\]
Thus $\vp u\psi_{+}=\chi_{+}$, or equivalently,
\[
\vp=\bar{u}\,\overline{\psi_{+}}\chi_{+}\qquad\text{a.e. on }\T.
	\]
Hence the factorization in Theorem~\ref{thm:analytic-Dphi} is satisfied.
	
Moreover,
\[
d=\gcd(u,\chi_{+})=\gcd(z,1)=1.
\]
Thus, in the notation of Theorem~\ref{thm:analytic-Dphi},
\[
u=du_1,\qquad\chi_{+}=d\chi_1,
\]
with
\[
u_1=z,\qquad\chi_1=1.
\]
Consequently,
\[
u\psi_{+}u_1H^2=z^2H^2\subseteq\cle_\vp^{(+)}.
\]
	
In fact, this inclusion is an equality. Indeed, let $h\in H^2$.
By Lemma~\ref{lem:proj-id},
\[
zh\oplus0\in\cle_\vp \quad\Longleftrightarrow\quad
P_{+}(\vp zh)\in zH^2.
\]
Since $\vp zh=\bar{z}\,zh=h$, we obtain
\[
zh\oplus0\in\cle_\vp\quad\Longleftrightarrow\quad
h\in zH^2.
\]
Writing $h=zg$ for some $g\in H^2$, we get $zh=z^2g$. Therefore,
\[
\cle_\vp^{(+)}=\{z^2g\oplus0:g\in H^2\}.
\]
	
For example, taking $g\equiv1$ gives
\[
f_0=z^2\oplus0\in\cle_\vp^{(+)}.
\]
Indeed,
\[
D_\vp f_0=P_{\clk_u^\perp}(\bar{z} z^2)=P_{\clk_u^\perp}(z)
=z,
\]
and hence
\[
\|D_\vp f_0\|=\|z\|=1=\|z^2\|=\|f_0\|.
\]
Thus, in this example,
\[
\cle_\vp^{(+)} =z^2H^2\oplus\{0\}.
\]
\end{example}

\begin{example}\label{ex:NA-u=z}
Let $u(z)=z$. Then the corresponding model space is $\clk_u=\mathrm{span}\{1\}$ and $\clk_u^\perp= zH^2\oplus H^2_{-}=\{f\in L^2:\widehat f(0)=0\}$.
Let $\vp(e^{it})=e^{i\psi(t)},\, t\in[0,2\pi]$ be any unimodular symbol (for example, $\psi$ may be taken as continuous and nonconstant). Since $|\vp|=1$ a.e. on $\T$, the multiplication operator $M_\vp$ is an isometry on $L^2$. Choose $0\neq h\in L^2$ with
\[\langle h,1\rangle=0 \quad\text{and}\quad \langle h,\bar{\vp}\rangle=0,\]
which is possible because $\operatorname{span}\{1,\bar{\vp}\}$ is at most two-dimensional while $L^2$ is infinite-dimensional.
Then $h\in \clk_u^\perp$ and
\[\langle M_\vp h,1\rangle=\langle h,\bar{\vp}\rangle=0,\]
so $M_\vp h\in \clk_u^\perp$ as well. Hence
\[D_\vp h=(I-P_{\clk_u})M_\vp h=M_\vp h,\]
and therefore,
\[\|D_\vp h\|=\|M_\vp h\|=\|h\|.\]
Thus $D_\vp$ attains its norm.

\end{example}

\begin{example}\label{ex:concrete-nonNA}
Let $u$ be any nonconstant inner function and define
\[
\vp(e^{it})=\frac{1+e^{it}}{2},\qquad t\in[0,2\pi).
\]
Then $\vp\in H^\infty$ and $\|\vp\|_\infty=1$. Moreover,
\[
|\vp(e^{it})|=\left|\frac{1+e^{it}}{2}\right|=|\cos(t/2)|<1
\]
for almost every $t\in[0,2\pi)$. Thus $|\vp|<1$ a.e. on $\T$.
	
We claim that $D_\vp$ is not norm attaining on $\clk_u^\perp$.
Indeed, let $0\neq h\in\clk_u^\perp$. Since $h\neq0$, the set
\[
E_h:=\{\zeta\in\T:|h(\zeta)|>0\}
\]
has positive measure. Because $|\vp|<1$ a.e.\ on $\T$, we have
\[
(1-|\vp|^2)|h|^2>0
\]
a.e. on $E_h$. Consequently,
\[
\|h\|^2-\|M_\vp h\|^2=\int_{\T}(1-|\vp|^2)|h|^2\,dm >0.
\]
Hence
\[
\|M_\vp h\|<\|h\|.
\]
Since $D_\vp=P_{\clk_u^\perp}M_\vp$ and $P_{\clk_u^\perp}$ is an orthogonal projection,
\[
\|D_\vp h\|\leq\|M_\vp h\|<\|h\|.
\]
Therefore, there is no nonzero $h\in\clk_u^\perp$ such that
$\|D_\vp h\|=\|h\|$. Since $\|D_\vp\|=\|\vp\|_\infty=1$, it follows that $D_\vp$ does not attain its norm.
\end{example}

\section{Essential maximum sets and finite Blaschke model spaces}
\label{sec:essential-maximum}

In the preceding section, several results were obtained under the
additional assumption that the symbol is unimodular. We now give a
description of the extremal vectors of $D_\vp$ for an arbitrary symbol $\vp\in L^\infty$. Throughout this section, set
\[
r:=\|\vp\|_\infty.
\]

\begin{lemma} \label{lem:exact-norm-defect}
For every $f\in\clk_u^\perp$,
\[
r^2\|f\|^2-\|D_\vp f\|^2=\int_{\T}(r^2-|\vp|^2)|f|^2\,dm
+\|P_{\clk_u}(\vp f)\|^2.
\]
\end{lemma}

\begin{proof}
Since $D_\vp f=P_{\clk_u^\perp}(\vp f)$ and $L^2=\clk_u\oplus\clk_u^\perp$, we have
\[
\|\vp f\|^2=\|P_{\clk_u}(\vp f)\|^2+\|D_\vp f\|^2.
\]
Consequently,
\[
\begin{aligned}
r^2\|f\|^2-\|D_\vp f\|^2
&=r^2\|f\|^2-\|\vp f\|^2+\|P_{\clk_u}(\vp f)\|^2\\
&=\int_{\T}(r^2-|\vp|^2)|f|^2\,dm +\|P_{\clk_u}(\vp f)\|^2.
\end{aligned}
\]
\end{proof}

The preceding identity gives an exact characterization of the entire
extremal space without any unimodularity assumption.

\begin{theorem}\label{thm:general-extremal-characterization}
Let $u$ be a nonconstant inner function and let $\vp\in L^\infty$.
Then
\[
\cle_\vp=\left\{f\in\clk_u^\perp:(r^2-|\vp|^2)|f|^2=0 \text{ a.e.},
	\quad\vp f\in\clk_u^\perp\right\}.
\]
Equivalently, $D_\vp$ is norm attaining if and only if there exists
$0\neq f\in\clk_u^\perp$ such that
\[
(r^2-|\vp|^2)|f|^2=0\quad\text{a.e.}\qquad\text{and}\qquad
P_{\clk_u}(\vp f)=0.
\]
\end{theorem}

\begin{proof}
Both terms on the right-hand side of Lemma~\ref{lem:exact-norm-defect} are nonnegative. Hence
$\|D_\vp f\|=r\|f\|$ if and only if both terms vanish. The first term vanishes if and only if $(r^2-|\vp|^2)|f|^2=0$ a.e., while the second vanishes if and only if $P_{\clk_u}(\vp f)=0$, equivalently,
$\vp f\in\clk_u^\perp$.
\end{proof}

\begin{remark}
If $|\vp|=\|\vp\|_\infty$ a.e., then the first condition in
Theorem~\ref{thm:general-extremal-characterization} is automatic.
Thus the theorem reduces to
\[
f\in\cle_\vp\quad\Longleftrightarrow\quad \vp f\in\clk_u^\perp,
\]
which recovers the unimodular extremal criterion from the preceding
section after normalization of the symbol.
\end{remark}

We next express the first condition in
Theorem~\ref{thm:general-extremal-characterization} entirely in terms
of the set on which the symbol attains its essential supremum.

Assume that $r>0$ and define
\[
E_\vp:=\{\zeta\in\T:|\vp(\zeta)|=r\},
\]
where $E_\vp$ is understood modulo sets of measure zero. On $E_\vp$
we have $|\vp|=r$ a.e., and hence $\vp/r$ is unimodular there.
Choose a unimodular function $\omega_\vp\in L^\infty$ such that
\[
\omega_\vp=\frac{\vp}{r}
\qquad\text{a.e. on }E_\vp.
\]
Outside $E_\vp$, define $\omega_\vp$ arbitrarily so that
$|\omega_\vp|=1$ a.e. on $\T$. The values of $\omega_\vp$ outside $E_\vp$ will play no role.

We identify $L^2(E_\vp)$ with the closed subspace of $L^2(\T)$
consisting of functions that vanish a.e.\ on $\T\setminus E_\vp$.

\begin{theorem}\label{thm:essential-maximum-characterization}
Let $\vp\in L^\infty$ be nonzero and $u$ be a nonconstant inner function. Then $\cle_\vp= L^2(E_\vp)\cap\clk_u^\perp \cap\overline{\omega_\vp}\clk_u^\perp.$
Moreover, if
\[
\cls_{\vp,u}:=\overline{\chi_{E_\vp}\clk_u
+\chi_{E_\vp}\overline{\omega_\vp}\clk_u}^{\,L^2(E_\vp)},
\]
then $\cle_\vp =L^2(E_\vp)\ominus\cls_{\vp,u}$. Consequently,
\[
D_\vp\text{ is norm attaining}\quad\Longleftrightarrow\quad
\cls_{\vp,u}\neq L^2(E_\vp).
\]
\end{theorem}

\begin{proof}
By Theorem~\ref{thm:general-extremal-characterization}, an extremal
vector $f$ must satisfy $(r^2-|\vp|^2)|f|^2=0$ a.e. This is equivalent to $f\in L^2(E_\vp)$.
	
For such an $f$, we have $\vp f=r\omega_\vp f$, and therefore
	\[
	\vp f\in\clk_u^\perp
	\quad\Longleftrightarrow\quad
	\omega_\vp f\in\clk_u^\perp
	\quad\Longleftrightarrow\quad
	f\in\overline{\omega_\vp}\clk_u^\perp.
	\]
This proves the first formula.
	
Now let $f\in L^2(E_\vp)$. Since $f$ vanishes outside $E_\vp$,
	\[
	f\in\clk_u^\perp
	\quad\Longleftrightarrow\quad
	f\perp\chi_{E_\vp}\clk_u
	\quad\text{in }L^2(E_\vp).
	\]
	Similarly,
	\[
	\omega_\vp f\in\clk_u^\perp
	\quad\Longleftrightarrow\quad
	f\perp
	\chi_{E_\vp}\overline{\omega_\vp}\clk_u
	\quad\text{in }L^2(E_\vp).
	\]
Hence $\cle_\vp=L^2(E_\vp)\ominus\cls_{\vp,u}$. The norm-attainment criterion follows immediately.
\end{proof}

\begin{corollary}\label{cor:essential-maximum-necessary}
If $\vp\in L^\infty$ be nonzero and $D_\vp$ is norm attaining, then
\[
m(E_\vp)>0.
\]
\end{corollary}

\begin{proof}
If $D_\vp$ is norm attaining, then Theorem~\ref{thm:essential-maximum-characterization} gives a nonzero
element of $L^2(E_\vp)$. Hence $m(E_\vp)>0$.
\end{proof}

The converse holds automatically when the model space is
finite-dimensional. This yields the following complete
characterization for finite Blaschke products.

\begin{theorem}\label{thm:finite-Blaschke-norm-attainment}
Let $u$ be a finite Blaschke product and let $\vp\in L^\infty$. Then
\[
D_\vp\text{ is norm attaining}\quad\Longleftrightarrow\quad
m\left(\{|\vp|=\|\vp\|_\infty\}\right)>0.
\]
If $\vp\neq0$ and $u$ has degree $n$, then whenever $D_\vp$ is norm
attaining,
\[
\operatorname{codim}_{L^2(E_\vp)}\cle_\vp=\dim\cls_{\vp,u}
\leq 2n.
\]
In particular, $\cle_\vp$ is infinite-dimensional.
\end{theorem}

\begin{proof}
Necessity follows from Corollary~\ref{cor:essential-maximum-necessary}. The case $\vp=0$ is
immediate, since then $D_\vp=0$ and $E_\vp=\T$.
	
Suppose now that $\vp\neq0$ and $m(E_\vp)>0$. If $u$ is a finite
Blaschke product of degree $n$, then $\dim\clk_u=n$. Hence
\[
\dim\cls_{\vp,u}\leq2n.
\]
Since Lebesgue measure on $\T$ is nonatomic and $m(E_\vp)>0$,
$L^2(E_\vp)$ is infinite-dimensional. Therefore, $\cls_{\vp,u}\neq L^2(E_\vp)$, and Theorem~\ref{thm:essential-maximum-characterization} implies that
$D_\vp$ is norm attaining.
	
The same theorem gives $\cle_\vp=L^2(E_\vp)\ominus\cls_{\vp,u}$, and therefore
\[
\operatorname{codim}_{L^2(E_\vp)}\cle_\vp =\dim\cls_{\vp,u}\leq2n.
\]
Thus $\cle_\vp$ is infinite-dimensional.
\end{proof}

Thus, for finite Blaschke model spaces, whether $D_\vp$ attains its norm depends only on the modulus of the symbol. The next consequence shows that full norm attainment may occur even when neither the analytic nor the coanalytic component contains a nonzero extremal vector.

\begin{corollary}\label{cor:characteristic-mixed-extremals}
Let $u$ be a finite Blaschke product and let $E\subset\T$ satisfy
\[
0<m(E)<1.
\]
Set $\vp=\chi_E$. Then $D_\vp$ is norm attaining, but
\[
\cle_\vp\cap uH^2=\{0\},\qquad \cle_\vp\cap H^2_{-}=\{0\}.
\]
Consequently, every nonzero extremal vector for $D_{\chi_E}$ has both
a nonzero analytic component and a nonzero coanalytic component.
\end{corollary}

\begin{proof}
Since $\|\chi_E\|_\infty=1$ and its essential maximum set is $E$, Theorem~\ref{thm:finite-Blaschke-norm-attainment} implies that
$D_{\chi_E}$ is norm attaining.
	
Suppose that $0\neq uh\in uH^2$ were extremal. By Theorem~\ref{thm:general-extremal-characterization},
\[
(1-\chi_E)|uh|^2=0
\qquad\text{a.e. on }\T.
\]
Since $u$ is inner and every nonzero $h\in H^2$ is nonzero a.e., $|uh|>0$ a.e. Hence the preceding equality would imply $m(\T\setminus E)=0$, contradicting $m(E)<1$. Therefore,
$\cle_\vp\cap uH^2=\{0\}$.
	
The same argument applies to a nonzero $g\in H^2_{-}$, because every
nonzero function in $H^2_{-}$ is nonzero almost everywhere. Hence
$\cle_\vp\cap H^2_{-}=\{0\}$.
	
Thus every nonzero extremal vector must have nonzero components in
both summands of $\clk_u^\perp=uH^2\oplus H^2_{-}$.
\end{proof}

\begin{remark}
Theorem~\ref{thm:finite-Blaschke-norm-attainment} and
Corollary~\ref{cor:characteristic-mixed-extremals} emphasize the
difference between full and componentwise norm attainment. For finite
Blaschke products, full norm attainment requires only that the symbol attain its essential supremum on a set of positive measure. In
contrast, norm attainment on either $uH^2$ or $H^2_{-}$ forces $|\vp|=\|\vp\|_\infty$ almost everywhere on $\T$.
\end{remark}

\begin{theorem}\label{thm:finite-Blaschke-componentwise}
Let $u$ be a finite Blaschke product and let $\vp\in L^\infty$.
Then the following are equivalent:
\begin{enumerate}
\item $D_\vp$ attains its norm on $uH^2$;
\item $D_\vp$ attains its norm on $H^2_{-}$;
\item $|\vp|=\|\vp\|_\infty$ a.e. on $\T$.

\end{enumerate}
If $u$ has degree $n$, then whenever these conditions hold, both $\cle_\vp^{(+)}$ and $\cle_\vp^{(-)}$ are infinite-dimensional,
and each has codimension at most $n$ in its corresponding component.
\end{theorem}

\begin{proof}
If $\vp=0$, then $D_\vp=0$, so every nonzero vector in $uH^2$ and in $H^2_{-}$ is extremal. Moreover, $|\vp|=\|\vp\|_\infty=0$ a.e.\ on $\T$. Thus all three conditions hold, and
\[
\cle_\vp^{(+)}=uH^2, \qquad \cle_\vp^{(-)}=H^2_{-},
\]
so both extremal spaces are infinite-dimensional and have codimension
zero in their corresponding components.
	
Assume henceforth that $r:=\|\vp\|_\infty>0$, and set $\omega:=\vp/r$.
	
\medskip
	
\noindent
\textbf{\textup{(i)} $\Rightarrow$ \textup{(iii)}.}
Suppose that $D_\vp$ attains its norm on $uH^2$. Since $D_\omega=r^{-1}D_\vp$, the operator $D_\omega$ also attains its norm
on $uH^2$. Moreover, $\|\omega\|_\infty=1$. Hence Proposition~\ref{prop:componentwise-unimodular} gives $|\omega|=1$ a.e.\ on $\T$. Therefore, $|\vp|=r=\|\vp\|_\infty$ a.e.\ on $\T$.
	
\medskip
	
\noindent
	\textbf{\textup{(ii)} $\Rightarrow$ \textup{(iii)}.}
The same argument applies if $D_\vp$ attains its norm on $H^2_{-}$.
Indeed, $D_\omega$ then attains its norm on $H^2_{-}$, and
Proposition~\ref{prop:componentwise-unimodular} again yields
$|\omega|=1$ a.e. Hence $|\vp|=\|\vp\|_\infty$ a.e.\ on $\T$.
	
	\medskip
	
\noindent
\textbf{\textup{(iii)} $\Rightarrow$ \textup{(i)} and \textup{(ii)}.}
Suppose that $|\vp|=r$ a.e.\ on $\T$. Then $\omega$ is unimodular.
	
Define
\[
T_{+}:H^2\longrightarrow\clk_u,\quad\text{by}\quad T_{+}h:=P_{\clk_u}(\omega uh).
\]
Since $u$ is a finite Blaschke product of degree $n$, we have $\dim\clk_u=n$. Thus
\[
\operatorname{codim}_{H^2}\cln (T_{+})=\operatorname{rank}T_{+}
	\leq n.
\]
Hence $\cln (T_{+})$ is infinite-dimensional. If $0\neq h\in\cln (T_{+})$, then $P_{\clk_u}(\omega uh)=0$. By Proposition~\ref{prop:extremal-set}, $uh$ is extremal for $D_\omega$.
Since $D_\vp=rD_\omega$, the same vector is extremal for $D_\vp$.
Therefore, $D_\vp$ attains its norm on $uH^2$.
	
Similarly, define
\[
T_{-}:H^2_{-}\longrightarrow\clk_u\quad\text{by}\quad
T_{-}g:=P_{\clk_u}(\omega g).
\]
Again, $\operatorname{codim}_{H^2_{-}}\cln (T_{-})
=\operatorname{rank}T_{-}\leq n$, so $\cln (T_{-})$ is infinite-dimensional. Every nonzero $g\in\cln (T_{-})$ satisfies
$P_{\clk_u}(\omega g)=0$, and hence is extremal for $D_\omega$, and therefore also for $D_\vp$. Thus $D_\vp$ attains its norm on $H^2_{-}$.
	
Finally,
\[
\cle_\vp^{(+)}=u\,\cln (T_{+}),\qquad \cle_\vp^{(-)}=\cln (T_{-}).
\]
Since multiplication by $u$ is unitary from $H^2$ onto $uH^2$, we
obtain
\[
\operatorname{codim}_{uH^2}\cle_\vp^{(+)} =\operatorname{codim}_{H^2}\cln (T_{+})\leq n,
\]
and
\[
\operatorname{codim}_{H^2_{-}}\cle_\vp^{(-)}=\operatorname{codim}_{H^2_{-}}\cln (T_{-})\leq n.
\]
Thus both componentwise extremal spaces are infinite-dimensional and have codimension at most $n$ in their corresponding components.
\end{proof}

\begin{theorem}[Density of norm attaining and non-norm attaining DTTOs]\label{thm:NA-density-finite-Blaschke}
Let $u$ be a finite Blaschke product. Then both sets
\[
\mathcal{N}_u:=\{\vp\in L^\infty:D_\vp\in\mathcal{NA}\}
\qquad\text{and}\qquad L^\infty\setminus\mathcal{N}_u
=\{\vp\in L^\infty:D_\vp\notin\mathcal{NA}\}
\]
are norm dense in $L^\infty$.
	
Consequently, both the norm attaining and the non-norm attaining
dual truncated Toeplitz operators are operator-norm dense in the class $\{D_\vp:\vp\in L^\infty\}$.
\end{theorem}

\begin{proof}
By Theorem~\ref{thm:finite-Blaschke-norm-attainment},
\[
D_\vp\in\mathcal{NA}\quad\Longleftrightarrow\quad m(E_\vp)>0,
\qquad E_\vp:=\{\zeta\in\T:|\vp(\zeta)|=\|\vp\|_\infty\}.
\]
	
We first prove that $\mathcal{N}_u$ is dense in $L^\infty$. Let $\vp\in L^\infty$ and $\varepsilon>0$. If $D_\vp\in\mathcal{NA}$, we may simply take $\psi=\vp$. Suppose, therefore, that $D_\vp\notin\mathcal{NA}$, and set $r:=\|\vp\|_\infty$. Since $D_0=0$ is norm attaining, necessarily $r>0$.
	
Set
\[
\eta:=\min\left\{\frac{\varepsilon}{2},\frac{r}{2}\right\}.
\]
By the definition of the essential supremum, the set $A_\eta:=\{\zeta\in\T:|\vp(\zeta)|>r-\eta\}$ has positive measure.
Since $\eta<r$, we also have $\vp\neq0$ a.e.\ on $A_\eta$. Define
\[
\psi(\zeta)
=\begin{cases}
		r\dfrac{\vp(\zeta)}{|\vp(\zeta)|},
		& \zeta\in A_\eta,\\[3mm]
		\vp(\zeta),
		& \zeta\notin A_\eta.
\end{cases}
\]
Then $\psi\in L^\infty$ and $\|\psi\|_\infty=r$. Moreover, for $\zeta\in A_\eta$,
\[
|\psi(\zeta)-\vp(\zeta)|=r-|\vp(\zeta)|<\eta,
\]
whereas $\psi=\vp$ on $\T\setminus A_\eta$. Hence $\|\psi-\vp\|_\infty\leq\eta<\varepsilon$.
	
Furthermore, $|\psi|=r$ a.e.\ on $A_\eta$, so $m(E_\psi)\geq m(A_\eta)>0$. Therefore, Theorem~\ref{thm:finite-Blaschke-norm-attainment} implies that
$D_\psi$ is norm attaining. Thus $\mathcal{N}_u$ is norm dense in
$L^\infty$.
	
We now prove that $L^\infty\setminus\mathcal{N}_u$ is norm dense in
$L^\infty$. Let $\vp\in L^\infty$ and $\varepsilon>0$. If $D_\vp\notin\mathcal{NA}$, we again take $\psi=\vp$. Suppose that
$D_\vp\in\mathcal{NA}$.
	
First assume that $\vp\neq0$, and set $r:=\|\vp\|_\infty>0$. By Theorem~\ref{thm:finite-Blaschke-norm-attainment}, $m(E_\vp)>0$. Since Lebesgue measure on $\T$ is nonatomic, there exists a measurable function $h:E_\vp\to(0,1]$ such that $\operatorname*{ess\,inf}_{E_\vp}h=0$. For instance, one may choose
pairwise disjoint measurable sets $E_n\subseteq E_\vp$ of positive
measure, define $h=1/n$ on $E_n$, and define $h=1$ on the remaining
part of $E_\vp$.
	
Choose
\[
0<\delta<\min\left\{1,\frac{\varepsilon}{r}\right\},
\]
and define
\[
\psi(\zeta)=
\begin{cases}
(1-\delta h(\zeta))\vp(\zeta),
& \zeta\in E_\vp,\\
\vp(\zeta),
& \zeta\notin E_\vp.
\end{cases}
\]
Then $\|\psi-\vp\|_\infty\leq\delta r<\varepsilon$.
	
We claim that $\|\psi\|_\infty=r$. Clearly $|\psi|\leq r$ a.e.
On the other hand, since $\operatorname*{ess\,inf}_{E_\vp}h=0$, for every $\eta>0$ the set $\{\zeta\in E_\vp:h(\zeta)<\eta\}$ has positive measure. On this set,
\[
|\psi(\zeta)|=r(1-\delta h(\zeta))>r(1-\delta\eta).
\]
Hence $\|\psi\|_\infty\geq r(1-\delta\eta)$ for every $\eta>0$.
Letting $\eta\downarrow0$ yields $\|\psi\|_\infty\geq r$, and therefore $\|\psi\|_\infty=r$.
	
However, on $E_\vp$ we have $h>0$ a.e., so $|\psi|=r(1-\delta h)<r$ a.e. On $\T\setminus E_\vp$, $|\psi|=|\vp|<r$ a.e. Consequently,
$|\psi|<r=\|\psi\|_\infty$ a.e.\ on $\T$, and hence $m(E_\psi)=0$. By
Theorem~\ref{thm:finite-Blaschke-norm-attainment}, $D_\psi$ is not norm attaining.
	
It remains to consider $\vp=0$. Given $\varepsilon>0$, choose
$0<c<\varepsilon$ and define
\[
\psi(z):=c\frac{1+z}{2},\, z\in \T.
\]
Then $\|\psi\|_\infty=c<\varepsilon$, while, for $z=e^{it}$,
\[
|\psi(e^{it})|=c\left|\cos\frac{t}{2}\right|<c
\]
for almost every $t$. Thus $m(E_\psi)=0$, and Theorem~\ref{thm:finite-Blaschke-norm-attainment} implies that
$D_\psi$ is not norm attaining. Therefore, $L^\infty\setminus\mathcal{N}_u$ is norm dense in $L^\infty$.
	
Finally, for $\vp,\psi\in L^\infty$, we have
\[
\|D_\vp-D_\psi\|=\|D_{\vp-\psi}\|=\|\vp-\psi\|_\infty.
\]
Thus the map $\vp\mapsto D_\vp$ is isometric. Hence the two density
statements above translate directly into operator-norm density of the
norm attaining and non-norm attaining DTTOs in the class $\{D_\vp:\vp\in L^\infty\}$.
\end{proof}

\section{A connection between norm attainment for TOs and DTTOs}

We conclude by relating norm attainment for classical Toeplitz operators to componentwise norm attainment for dual truncated Toeplitz operators. We first recall Yoshino's characterization.

\begin{theorem}\cite{Yoshino2002}\label{thm:Toeplitz-NA}
Let $\vp\in L^\infty$. Then $T_\vp$ is norm attaining if and only if
there exist inner functions $\Theta_1$ and $\Theta_2$, with no common
inner factor, such that
\[
\vp=\|\vp\|_\infty\Theta_1\overline{\Theta_2} \qquad\text{a.e. on }\T.
\]
\end{theorem}

Combining this factorization with
Theorem~\ref{thm:quotient-inner-extremals} gives the following
connection.

\begin{corollary}\label{cor:Toeplitz-implies-DTTO}
Let $\vp\in L^\infty$. If $T_\vp$ is norm attaining, then $D_\vp$
attains its norm on both $uH^2$ and $H^2_{-}$. More precisely, if
\[
\vp=\|\vp\|_\infty\Theta_1\overline{\Theta_2},
\]
where $\Theta_1$ and $\Theta_2$ are inner, then
\[
u\Theta_2H^2\subseteq\cle_\vp^{(+)},
	\qquad
V(\Theta_1H^2)\subseteq\cle_\vp^{(-)}.
\]
In particular, $D_\vp$ is norm attaining.
\end{corollary}

\begin{proof}
The case $\vp=0$ is immediate. Suppose that $\vp\neq0$. By Theorem~\ref{thm:Toeplitz-NA},
\[
\vp=\|\vp\|_\infty\overline{\Theta_2}\Theta_1.
\]
The conclusion now follows directly from Theorem~\ref{thm:quotient-inner-extremals} with $a=\Theta_2$ and $b=\Theta_1$.
\end{proof}

The converse holds under the natural analytic or coanalytic
assumption on the symbol.

\begin{corollary}\label{cor:TO-DTTO-componentwise}
Let $\vp\in L^\infty$. Then the following statements hold:
\begin{enumerate}
\item If $\vp\in H^\infty$, then $T_\vp\in\mathcal{NA}$ if and only if $D_\vp$ attains its norm on $uH^2$.

\item If $\bar{\vp}\in H^\infty$, then $T_\vp\in\mathcal{NA}$ if and only if $D_\vp$ attains its norm on $H^2_{-}$.

\end{enumerate}
\end{corollary}

\begin{proof}
Suppose first that $\vp\in H^\infty$. The forward implication follows
from Corollary~\ref{cor:Toeplitz-implies-DTTO}. Conversely, suppose
that $D_\vp$ attains its norm on $uH^2$. Then there exists $0\neq h\in H^2$ such that
\[
\|D_\vp(uh)\|=\|\vp\|_\infty\|h\|.
\]
Since $\vp\in H^\infty$, we have $\vp uh\in H^2$. Note that
\[
\|D_\vp(uh)\|^2 = \|(I-P_{+})(\vp uh)\|^2 +\|uP_{+}(\bar u\,\vp uh)\|^2,
\]
 where the first term vanishes, while $P_{+}(\bar{u}\vp uh) = P_{+}(\vp h)=T_\vp h$. Hence
\[
\|D_\vp(uh)\|^2 = \|uT_\vp h\|^2 = \|T_\vp h\|^2,
\]
because multiplication by the inner function $u$ is an isometry on
$H^2$. Therefore,
\[
\|T_\vp h\|=\|D_\vp(uh)\|=\|\vp\|_\infty\|h\|=\|T_\vp\|\,\|h\|.
\]
Thus $h$ is a nonzero extremal vector for $T_\vp$, and consequently
$T_\vp$ is norm attaining.
	
Now suppose that $\bar{\vp}\in H^\infty$. The forward implication again follows from Corollary~\ref{cor:Toeplitz-implies-DTTO}. Conversely, if $D_\vp$ attains its norm on $H^2_{-}$, then, since
\[
D_{\bar{\vp}}=C_uD_\vp C_u \qquad\text{and}\qquad C_u(H^2_{-})=uH^2,
\]
the operator $D_{\bar{\vp}}$ attains its norm on $uH^2$. By the first
part, $T_{\bar{\vp}}$ is norm attaining. Since $T_{\bar{\vp}}=T_\vp^*$, Theorem~\ref{thm:NA_equiv} implies that	$T_\vp$ is norm attaining.
\end{proof}

\section*{Declarations}

\noindent\textbf{Funding.}
Not applicable.

\medskip

\noindent\textbf{Conflict of interest.}
The authors declare that there is no conflict of interest regarding the publication of this manuscript.

\medskip

\noindent\textbf{Data availability.}
Not applicable.

\section*{Acknowledgment}
The authors sincerely thank their Ph.D. supervisor, Professor G. Ramesh (IIT Hyderabad), for his careful reading of the manuscript and for several valuable suggestions that improved both the presentation and the mathematical content of the paper.

\end{document}